\documentclass[a4paper,10pt,leqno]{amsart}
\title{Universal $L^2$-torsion, polytopes and  applications to $3$-manifolds}

\author{Stefan Friedl}
\address{Fakult\"at f\"ur Mathematik\\ Universit\"at Regensburg\\93040 Regensburg\\   Germany}
\email{sfriedl@gmail.com}

\author{Wolfgang L\"uck}
        \address{Mathematisches Institut der Universit\"at Bonn\\
                Endenicher Allee 60\\
                53115 Bonn, Germany}
         \email{wolfgang.lueck@him.uni-bonn.de}
          \urladdr{http://www.him.uni-bonn.de/lueck}

         \date{September   2016.}

\usepackage{hyperref}
\usepackage{color}
\usepackage{pdfsync}
\usepackage{calc}
\usepackage{enumerate,amssymb,enumitem}
\usepackage[arrow,curve,matrix,tips,2cell]{xy}
  \SelectTips{eu}{10} \UseTips
  \UseAllTwocells

\DeclareMathAlphabet\EuR{U}{eur}{m}{n}
\SetMathAlphabet\EuR{bold}{U}{eur}{b}{n}

\makeindex             

\theoremstyle{plain}
\newtheorem{theorem}{Theorem}[section]
\newtheorem{lemma}[theorem]{Lemma}
\newtheorem{proposition}[theorem]{Proposition}

\theoremstyle{definition}

\newtheorem{definition}[theorem]{Definition}
\newtheorem{example}[theorem]{Example}

\newtheorem{remark}[theorem]{Remark}

{\catcode`@=11\global\let\c@equation=\c@theorem}

\newcommand{\comsquare}[8]                   
{\begin{CD}
#1 @>#2>> #3\\
@V{#4}VV @V{#5}VV\\
#6 @>#7>> #8
\end{CD}
}

\newcommand{\xycomsquare}[8]                   
{\xymatrix
{#1 \ar[r]^{#2} \ar[d]^{#4} &
#3 \ar[d]^{#5}  \\
#6\ar[r]^{#7} &
#8
}
}

\newcommand{\xycomsquareminus}[8]                      
{\xymatrix{#1 \ar[r]^-{#2} \ar[d]^-{#4} &
#3 \ar[d]^-{#5}  \\
#6\ar[r]^-{#7} &
#8
}
}

\newcommand{\cald}{\mathcal{D}}

\newcommand{\calsn}{\mathcal{SN}}

\newcommand{\calb}{\mathcal{B}}
\newcommand{\calc}{\mathcal{C}}

\newcommand{\caln}{{\mathcal N}}
\newcommand{\calp}{\mathcal{P}}
\newcommand{\calr}{\mathcal{R}}
\newcommand{\calu}{\mathcal{U}}

\newcommand{\IC}{{\mathbb C}}

\newcommand{\IN}{{\mathbb N}}

\newcommand{\IP}{{\mathbb P}}
\newcommand{\IQ}{{\mathbb Q}}
\newcommand{\IR}{{\mathbb R}}

\newcommand{\IZ}{{\mathbb Z}}

\newcommand{\bfP}{{\mathbf P}}

\newcommand{\abel}{\operatorname{abel}}

\newcommand{\ch}{\operatorname{ch}}

\newcommand{\coker}{\operatorname{coker}}

\newcommand{\cone}{\operatorname{cone}}

\newcommand{\el}{\operatorname{el}}
\newcommand{\ev}{\operatorname{ev}}
\newcommand{\Ext}{\operatorname{Ext}}

\renewcommand{\hom}{\operatorname{Hom}}

\newcommand{\id}{\operatorname{id}}

\newcommand{\im}{\operatorname{im}}

\newcommand{\odd}{\operatorname{odd}}

\newcommand{\orb}{\operatorname{orb}}
\newcommand{\poly}{\operatorname{poly}}
\newcommand{\pr}{\operatorname{pr}}

\newcommand{\Rep}{\operatorname{Rep}}

\newcommand{\rk}{\operatorname{rk}}

\newcommand{\sn}{\operatorname{sn}}

\newcommand{\supp}{\operatorname{supp}}

\newcommand{\tors}{\operatorname{tors}}

\newcommand{\Wh}{\operatorname{Wh}}

\newcommand{\smfrac}[2]{\mbox{\footnotesize$\displaystyle\frac{#1}{#2}$}} 
\newcommand{\tmfrac}[2]{\mbox{\large$\frac{#1}{#2}$}} 

\newcommand{\smsum}[2]{\mbox{\footnotesize$\displaystyle\sum\limits_{#1}^{#2}$}} 
\newcommand{\tmsum}[2]{\mbox{$\textstyle \sum\limits_{#1}^{#2}$}} 

\newcommand{\higherlim}[3]{{\setbox1=\hbox{\rm lim}
        \setbox2=\hbox to \wd1{\leftarrowfill} \ht2=0pt \dp2=-1pt
        \mathop{\vtop{\baselineskip=5pt\box1\box2}}
        _{#1}}^{#2}#3}

\newcommand{\version}[1]                       
{\begin{center} last edited on #1\\
last compiled on \today\\
name of texfile: \jobname
\end{center}
}

\newcounter{commentcounter}

\begin{document}

\typeout{----------------------------  localization.tex  ----------------------------}


\typeout{------------------------------------ Abstract ----------------------------------------}

\begin{abstract}
  Given an $L^2$-acyclic connected finite $CW$-complex, we define its universal
  $L^2$-torsion in terms of the chain complex of its universal covering. It takes values
  in the weak Whitehead group $\Wh^w(G)$. We study its main properties such as homotopy
  invariance, sum formula, product formula and Poincar\'e duality.  Under certain assumptions, 
we can specify certain homomorphisms from the weak Whitehead
  group $\Wh^w(G)$ to abelian groups such as the real numbers or the Grothendieck group of
  integral polytopes, and the image of the universal $L^2$-torsion can be identified with
  many invariants such as the $L^2$-torsion, the $L^2$-torsion function, twisted
  $L^2$-Euler characteristics and, in the case of a $3$-manifold, the dual Thurston norm
  polytope.
\end{abstract}

\maketitle


 \typeout{-------------------------------   Section 0: Introduction --------------------------------}

\setcounter{section}{-1}
\section{Introduction}

We assign to a group $G$ the weak $K_1$-groups $K_1^w(\IZ G)$, $\widetilde{K}_1^w(\IZ G)$
and the weak Whitehead group $\Wh^w(G)$, see Definition~\ref{def:K_1_upper_w(ZG)}. These
groups are, as the names suggest, variations on the corresponding classical groups. More
precisely, the name `weak' comes from the fact that in the definitions we study matrices
over the group ring $\IZ G$, but we no longer require that they are
invertible over the group ring, but we only require that they are weakly invertible in the
sense of $L^2$-invariants.  These groups are in general much larger than their classical
analogues. For instance, for $G = \IZ$ we have $\Wh(\IZ)=0$ but $\Wh^w(\IZ)\cong
\IQ(z^{\pm 1})^{\times}/\{\pm z^n \mid n \in \IZ\}$.

Furthermore, to an $L^2$-acyclic
finite based free $\IZ G$-chain complex $C_*$ we assign its \emph{universal $L^2$-torsion}, see
Definition~\ref{def:Universal_L2-torsion_for_L2-acyclic_finite_based_free_ZG-chain_complexes}
\begin{equation}
  \rho^{(2)}_u(C_*) \in \widetilde{K}_1^w(G).
\end{equation}
It is characterized by the universal properties that  \[\rho^{(2)}_u\big(0\to \IZ G\xrightarrow{\pm \id}\IZ G\to 0\big) = 0\]
 and that for any short based exact sequence $0 \to C_* \to D_* \to E_* \to 0$ we get 
$\rho^{(2)}_u(D_*) =\rho^{(2)}_u(C_*) + \rho^{(2)}_u(E_*)$, as explained in
Remark~\ref{rem:universal_property_of_the_universal_L2-torsion}. If $X$ is an $L^2$-acyclic finite
free $G$-$CW$-complex, it defines an element
\begin{equation}
  \rho^{(2)}_u(X;\caln(G)) \in \Wh^w(G):=\widetilde{K}_1^w(G)/\pm G,
\end{equation}
determined by $\rho^{(2)}_u(C_*(X))$, where $C_*(X)$ is the cellular $\IZ G$-chain complex
of $X$, see Definition~\ref{def:universal_L2-torsion_for_G-CW-complexes}.  As an example,
if $X$ is the free abelian cover of a knot exterior $S^3\setminus \nu K$, then under the
aforementioned isomorphism $\Wh^w(\IZ)\cong \IQ(z^{\pm 1})^{\times}/\{\pm z^n \mid n \in
\IZ\}$ the universal torsion $\rho^{(2)}_u(X;\caln(\IZ))$ agrees with the Turaev-Milnor
torsion~\cite{Milnor(1968a),Turaev(1986),Turaev(2001)} of the knot.

If $X$ is a finite
connected $CW$-complex with fundamental group $\pi$, and universal covering
$\widetilde{X}$, we call $X$ $L^2$-acyclic if $\widetilde{X}$ is $L^2$-acyclic as finite
free $\pi$-$CW$-complex and abbreviate
\begin{equation}
  \rho^{(2)}_u(\widetilde{X}) := \rho^{(2)}_u(\widetilde{X};\caln(\pi)) \in \Wh^w(\pi).
\end{equation}

The basic properties of these invariants including homotopy invariance, sum formula,
product formula, and Poincar\'e duality are collected in
Theorem~\ref{the:Main_properties_of_the_universal_L2-torsion} and
Theorem~\ref{the:Main_properties_of_the_universal_L2-torsion_for_universal_coverings}. One
can show for a finitely presented group $G$, for which there exists at least one
$L^2$-acyclic finite connected $CW$-complex $X$ with $\pi_1(X) = G$, that every element in
$\Wh^w(G)$ can be realized as $\rho^{(2)}_u(\widetilde{Y})$ for some $L^2$-acyclic
finite connected $CW$-complex $Y$ with $G = \pi_1(Y)$, see
Lemma~\ref{lem:realizability_G-CW}. 

The point of this new invariant is that it encompasses many other well-known invariants.
We illustrate this by considering some examples. Although these invariants do make sense
in all dimensions, we often restrict ourselves to the case of admissible $3$-manifolds $M$:

\begin{definition}[Admissible $3$-manifold] \label{def:admissible_3-manifold} 
   A $3$-manifolds is called \emph{admissible} if it is compact, connected, orientable, and irreducible,
  its boundary is empty or a disjoint union of tori, and its fundamental group
  is infinite.
\end{definition}


\subsection{$L^2$-torsion and the $L^2$-Alexander torsion}
In Section~\ref{subsec:L2-torsion} we will see that the universal $L^2$-torsion of a
$3$-manifold $M$ determines the $L^2$-torsion and more generally the $L^2$-torsion function
(also called $L^2$-Alexander polynomial and $L^2$-Alexander torsion) that recently was
intensively studied, see e.g.,
\cite{BenAribi(2016),Dubois-Friedl-Lueck(2016Alexander),Dubois-Wegner(2015),Herrmann(2016),
Li-Zhang(2006volume),Li-Zhang(2006Alexander)}. The
former invariant is determined by the volumes of the hyperbolic pieces in the
Jaco-Shalen-Johannson decomposition of $M$,
see~\cite[Theorem~0.7]{Lueck-Schick(1999)}. The latter invariant is a function
$\rho^{(2)}(\widetilde{M},\phi)\colon \IR_{>0}\to \IR_{\geq 0}$ associated to $M$ and a class $\phi\in H^1(M;\IZ)$.  This
function captures a lot of interesting topological information, in particular it was shown
by the authors~\cite{Friedl-Lueck(2015l2+Thurston)} and independently by
Liu~\cite{Liu(2015)} that it determines the Thurston norm of $\phi$.


\subsection{Twisted $L^2$-Euler characteristic}
  \label{subsec:twisted_L2-Euler_characteristic}
In Section~\ref{subsec:The_polytope_homomorphism} we recall the statement of the Atiyah Conjecture. 
  For a short discussion of the Atiyah Conjecture for a torsionfree group $G$ and
  its status which is adapted to the needs of this paper, we refer
  to~\cite[Section~3]{Friedl-Lueck(2016l2+poly)}, a more general introduction is given
  in~\cite[Chapter~10]{Lueck(2002)}.  In Remark~\ref{rem:pairing_and_introduction}, 
  given  a torsionfree group $G$ satisfying the Atiyah  Conjecture, we will introduce a pairing
  \begin{equation}
  \Wh^w(\pi) \times H^1(G) \to \IZ,
  \label{pairing_wh_otimes_H1_to_Z}
  \end{equation}
  which has the following property: given an $L^2$-acyclic finite free $\IZ G$-chain complex $C_*$ and an element 
  $\phi \in   H^1(G) = \hom(G,\IZ)$ the image of
  $\rho^{(2)}_u(C_*;\caln(G)) \otimes \phi$ under the pairing above equals the $L^2$-Euler characteristic 
we introduced in~\cite{Friedl-Lueck(2016l2+poly)}.
  
   It is now easy to see that for an admissible $3$-manifold $M$ and
  $\phi\colon \pi_1(M)\to \IZ$ the $L^2$-Euler characteristic of
  \cite{Friedl-Lueck(2016l2+poly)} is determined by the universal $L^2$-torsion
  $\rho_u^{(2)}(\widetilde{M})$. Similarly, using~\cite[Section~8]{Friedl-Lueck(2016l2+poly)} 
  one can show that the degrees of the
  higher-order Alexander polynomials of Cochran~\cite{Cochran(2004)} and
  Harvey~\cite[Section~3]{Harvey(2005)} of knots and $3$-manifolds are determined by
  universal $L^2$-torsions of appropriate covers of $M$. 


\subsection{Polytopes}
  \label{subsec:polytopes}
  For a finitely generated free abelian group $H$ we denote by $\calp_{\IZ}^{\Wh}(H)$ the abelian group given
  by the Grothen\-dieck construction applied to the abelian monoid of integral polytopes
  in $\IR \otimes_{\IZ} H$ under the Minkowski sum, modulo translation with elements in
  $H$, see~\eqref{calp_IZ,Wh}. Given a group $G$ that satisfies the Atiyah Conjecture we
  will construct the \emph{polytope homomorphism}
\[
\IP \colon \Wh^w(\IZ G) \to \calp_{\IZ}^{\Wh}(H_1(G)_f)
\]
in~\eqref{polytope_homomorphism_Wh} and define the \emph{$L^2$-torsion polytope}
$P(\widetilde{M})$ to be the image of the universal $L^2$-torsion under the polytope homomorphism.

Of particular interest is the composition of the universal torsion with the polytope
homomorphism. For example let $M$ be an admissible $3$-manifold that is not a closed graph manifold. Then we
obtain a well-defined element
\[
P(\widetilde{M})\,:=\, \IP(\rho_u(\widetilde{M}))\,\in \, \calp_{\IZ}^{\Wh}(H_1(M)_f).
\]
We now explain the information contained in this invariant. Following Thurston~\cite{Thurston(1986)} we can
use the minimal complexity of surfaces in a $3$-manifold $M$ to assign to $M$ an integral
polytope $T(M) \subseteq \IR \otimes_{\IZ} H_1(\pi)_f$, called the \emph{dual Thurston
  polytope}, see~\eqref{dual_Thurston_polytope}  for details. One of our main results is that we show in
Theorem~\ref{the:dual_Thurston_polytope_and_the_L2-torsion_polytope_new} that 
\begin{eqnarray*} 
[T(M)^*] & = & 2 \cdot  P(\widetilde{M})\, \in \, \calp_{\IZ}^{\Wh}(H_1(\pi)_f).
\end{eqnarray*}
We can also
use this approach to assign formal differences of polytopes to many other groups, e.g.\
free-by-cyclic groups and two-generator one-relator groups. We will discuss these examples
in more details in a joint paper~\cite{Friedl-Lueck-Tillmann(2016)} with Stephan Tillmann.

Finally let $G$ be any group that admits a finite model for $BG$ and that satisfies the
Atiyah Conjecture and let $f\colon G\to G$ be a monomorphism. Then we can associate to
this monomorphism the polytope invariant of the corresponding ascending HNN-extension. If
$G=F_2$ is the free group on two generators this polytope invariant has been studied by
Funke--Kielak~\cite{Funke-Kielak(2016)}. We hope that this invariant of monomorphisms of
groups will have other interesting applications.


\subsection*{Acknowledgments.}
We wish to thank the referee for carefully reading the first version of the paper and for making many very helpful remarks.

The first author gratefully acknowledges the support provided by the SFB 1085 ``Higher
Invariants'' at the University of Regensburg, funded by the Deutsche
Forschungsgemeinschaft {DFG}.  The paper is financially supported by the Leibniz-Preis of
the second author granted by the {DFG} and the ERC Advanced Grant ``KL2MG-interactions''
(no.  662400) of the second author granted by the European Research Council.

\tableofcontents


 \typeout{------------------   Section 1: Universal $L^2$-torsion for chain complexes-------------------}

\section{Universal $L^2$-torsion for chain complexes}
\label{sec:Universal_torsion_for_chain_complexes}

In this section we define the universal $L^2$-torsion and the algebraic $K$-group where it
takes values in.  

%
%


\subsection{From finite based free $\IZ G$-chain complexes to finite Hilbert $\caln(G)$-chain complexes}
\label{subsec:From_finite_based_free_ZG-chain_complexes_to_finite_Hilbert_caln(G)-chain_complexes}

Let $G$ be a group.  We will always work with left modules.
A \emph{$\IZ G$-basis} for a finitely generated free $\IZ G$-module
$M$ is a finite subset $B \subseteq M$ such that the canonical map 
\[
\bigoplus_{b \in B} \IZ G \to  M, \quad \; (r_b)_{b \in B} 
\mapsto \sum_{b \in B} r_b \cdot b
\]
 is a $\IZ G$-isomorphism.  Notice that we do not
require that $B$ is ordered.  A \emph{finitely generated based free $\IZ G$-module} is
a pair $(M,B)$ consisting of a finitely generated free $\IZ G$-module $M$ together with a
basis $B$.  (Sometimes we omit $B$ from the notation.) In the sequel we will equip 
$\IZ G^n = \bigoplus_{i = 1}^n \IZ G$ with the standard basis.
Throughout this paper a basis is understood to be unordered.

A \emph{finite Hilbert $\caln(G)$-module} $V$ is a Hilbert space $V$ together with a linear
isometric left $G$-action such that there exists an isometric linear embedding into $L^2(G)^n$
for some natural number $n$.  (Here $\caln(G)$ stands for the group von Neumann algebra
which is the algebra of bounded $G$-equivariant operators $L^2(G) \to L^2(G)$.)
Obviously $L^2(G)^n$ is a finite Hilbert
$\caln(G)$-module. Morphisms of finite Hilbert $\caln(G)$-modules are bounded
$G$-equivariant operators (which are not necessarily isometric).  
For a basic introduction to Hilbert $\caln(G)$-chain complexes and $L^2$-Betti
numbers we refer for instance to~\cite[Chapter~1]{Lueck(2002)}.
For a finitely generated  based free $\IZ G$-module $(M,B)$ define a finite Hilbert
$\caln(G)$-module $\Lambda^G(M)$ by $L^2(G) \otimes_{\IZ G} M$ equipped with the Hilbert
space structure, for which the induced map
\[
\bigoplus_{b \in B} L^2(G) \to L^2(G) \otimes_{\IZ G} M, 
\quad (x_b)_{b \in B} \mapsto \sum_{b \in B} x_b \otimes_{\IZ G} b
\]
is a $G$-equivariant isometric invertible operator of Hilbert spaces. (If $G$ is clear
from the context, we often abbreviate $\Lambda^G$ by $\Lambda$.)  Obviously $\Lambda(\IZ G^n)$ 
can be identified with $L^2(G)^n$.  One easily checks that a $\IZ G$-map of finitely
generated based free $\IZ G$-modules $f \colon M \to N$ induces a morphism $\Lambda(f)$ of Hilbert
$\caln(G)$-modules. This construction is functorial. Moreover we have
$r_1 \cdot \Lambda(f_1) + r_2 \cdot \Lambda (f_2) = \Lambda(r_1 \cdot f_0 + r_2 \cdot f_2)$
for integers $r_1, r_2$ and $\IZ G$-maps $f_0, f_1 \colon M \to N$ of finitely
generated based free $\IZ G$-modules.  Obviously $\Lambda$ is compatible with direct sums,
i.e., for two finitely generated based free $\IZ G$-modules $M_0$ and $M_1$ the canonical
map $\Lambda(M_0) \oplus \Lambda(M_1) \to \Lambda(M_0 \oplus M_1)$
is an isometric $G$-equivariant invertible operator.  Let $f \colon M \to N$ be a $\IZ
G$-homomorphism of finitely generated based free $\IZ G$-modules. If $B$ and $C$ are the
bases of $M$ and $N$, let $A(f)$ be the $B$-$C$-matrix defined by $f$, namely, if for $b \in
B$ we write $f(b) = \sum_{c \in C} a_{b,c} \cdot c$ for $a_{b,c} \in \IZ G$, then $A =
(a_{b,c})$.  Let $A^*$ be the $C$-$B$-matrix, whose entry at $(c,b) \in C \times B$ is
$\overline{a_{b,c}}$, where $\overline{\sum_{g \in G} r_g \cdot g} := \sum_{g \in G} r_g \cdot g^{-1}$ 
for $\sum_{g \in G} r_g \cdot g \in \IZ $.  Define $f^* \colon N \to M$ to
be the $\IZ G$-homomorphism associated to $A^*$. One easily checks that $\Lambda(f^*) =
\Lambda(f)^*$, where $\Lambda(f)^*$ is the adjoint operator associated to $\Lambda(f)$.
All in all one may summarize by saying that $\Lambda$ defines a functor of additive categories with
involution from the category of finitely generated based free $\IZ G$-modules to the
category of finite Hilbert $\caln(G)$-modules.

Given an $(m,n)$-matrix $A = (a_{i,j})$, where $i \in \{1,2, \ldots, m\}$ and $j \in \{1,2, \ldots, n\}$,
right multiplication with $A$ defines a $\IZ G$-homomorphism 
\[
r_A \colon \IZ G^m \to \IZ G^n, \quad (x_1, x_2, \ldots , x_m) 
\mapsto (x_1, x_2, \ldots , x_m) \cdot A = \Big(\, \tmsum{i = 1}{m} x_i \cdot a_{i,j}\Big)_{j = 1,2, \ldots, n}
\]
Notice that $A$ is the matrix $A(r_A)$ associated to $r_A$. 

A finite based free $\IZ G$-chain complex $C_*$ is a $\IZ G$-chain complex $C_*$ such that
there exists a natural number $N$ with $C_n = 0$ for $|n| > N$ and such that each chain
module $C_n$ is a finitely generated based free $\IZ G$-module.  Denote by $\Lambda(C_*)$
the finite Hilbert $\caln(G)$-chain complex which is obtained by applying the functor
$\Lambda$. Here by a finite Hilbert $\caln(G)$-chain complex $D_*$ we mean a chain complex
in the category of finite Hilbert $\caln(G)$-chain complexes such that there exists a
natural number $N$ with $D_n = 0$ for $n > N$.

\subsection{The weak $K_1$-group $K_1^w(\IZ G)$}
\label{subsec:K_1_upper_w(ZG)}

The universal $L^2$-torsion, that we will define in
Section~\ref{subsec:The_universal_L2-torsion_for_chain_complexes}, will take values in the
following $K_1$-group.  Recall that a morphism $f \colon V \to W$ of finite Hilbert
$\caln(G)$-modules is called a \emph{weak isomorphism} if it is injective and has dense
image.

\begin{definition}[$K^w_1(\IZ G)$]
\label{def:K_1_upper_w(ZG)} 
Define  the \emph{weak $K_1$-group}
\[
K_1^w(\IZ G)
\]
to be the abelian group defined in terms of generators and relations as follows.
Generators $[f]$ are given by  of $\IZ G$-endomorphisms 
$f \colon \IZ G^n \to \IZ G^n$ for $n \in \IZ, n \ge 0$ such that $\Lambda(f)$ is a weak isomorphism of
finite Hilbert $\caln(G)$-modules.  If $f_1,f_2 \colon \IZ G^n \to \IZ G^n$ are $\IZ G$-endomorphisms 
such that $\Lambda(f_1)$ and $\Lambda(f_2)$ are weak isomorphisms, then
we require the relation
\[
[f_2 \circ f_1] = [f_1] + [f_2].
\]
If $f_0 \colon \IZ G^m \to \IZ G^m$, $f_2 \colon \IZ G^n \to \IZ G^n$ and $f_1 \colon \IZ G ^n \to \IZ G^m$ 
are $\IZ G$-maps such that $\Lambda(f_0)$ and $\Lambda(f_2)$ are weak isomorphisms, then 
we require for the $\IZ G$-map 
\[
f = \begin{pmatrix} f_0 & f_1 \\ 0 & f_2 \end{pmatrix} \colon \IZ G^{m + n} = \IZ G^m \oplus \IZ G^n  
\to \IZ G^m \oplus \IZ G^n
\]
the relation
\[
[f] = [f_0] + [f_2].
\]
 Let 
\[
\widetilde{K}^w_1(\IZ G)
\]
be the quotient of $K_1^w(\IZ G)$ by the subgroup generated by
the element $[- \id \colon \IZ G\to \IZ G]$.
This is the same as the cokernel of the obvious
composite $K_1(\IZ) \to K_1^w(\IZ G)$.
Define the \emph{weak Whitehead group} of $G$ 
\[
\Wh^w(G)
\]
to be the cokernel of the homomorphism
\[
\{\sigma \cdot g \mid \sigma \in \{ \pm 1\}, g \in G\} \to K_1^w(\IZ G), 
\quad 
\sigma \cdot g \mapsto [r_{\sigma \cdot g} \colon \IZ G \to \IZ G].
\]
\end{definition}

Definition~\ref{def:K_1_upper_w(ZG)}  makes sense since the morphisms appearing above
$\Lambda(f_2 \circ f_2)$ and $\Lambda(f)$ are again weak isomorphism
by~\cite[Lemma~3.37 on page~144]{Lueck(2002)}.

\begin{remark}
We obtain the classical notions of  $K_1(\IZ G)$ and $\widetilde{K}_1(\IZ G)$ if we replace in
Definition~\ref{def:K_1_upper_w(ZG)}  above everywhere the condition that $\Lambda(f)$ is a weak isomorphism of
finite Hilbert $\caln(G)$-modules by the stronger condition that $f$ is a $\IZ G$-isomorphism.
\end{remark}

\subsection{The universal $L^2$-torsion for chain complexes}
\label{subsec:The_universal_L2-torsion_for_chain_complexes}

Recall that a $\IZ G$-chain complex $C_*$ is \emph{contractible}, if $C_*$ possesses a \emph{chain contraction}
$\gamma_*$, i.e., a sequence of $\IZ G$-maps $\gamma_n \colon C_n \to C_{n+1}$
such that $c_{n+1} \circ \gamma_n + \gamma_{n-1} \circ c_n = \id_{C_n}$ holds for all $n \in \IZ$. 
Moreover, for a finite based free contractible $\IZ G$-chain complex $C_*$ its \emph{(classical) torsion}
\begin{eqnarray}
\rho(C_*) \in \widetilde{K}_1(\IZ G)
\label{classical_torsion}
\end{eqnarray}
is defined for any choice of chain contraction $\gamma_*$ by the class $[(c + \gamma)_{\odd}]$
of the $\IZ G$-isomorphism  of finitely generated based free $\IZ G$-modules 
$(c + \gamma)_{\odd} \colon C_{\odd} \xrightarrow{\cong} C_{\ev}$, where here and in the 
sequel we write
\[
C_{\odd} 
\,\, = \,\, 
\bigoplus_{n \in \IZ} C_{2n+1} \quad \mbox{ and }\quad 
C_{\ev} 
\,\, = \,\,
\bigoplus_{n \in \IZ} C_{2n}.
\]
As we mentioned above, a basis is understood to be unordered. 
This implies that the class
 $[(c + \gamma)_{\odd}]$ is not well-defined in ${K}_1(\IZ G)$,
 but it is indeed well-defined in  $\widetilde{K}_1(\IZ G)={K}_1(\IZ G)/[-\id]$.

We want to carry over these two notions to finite based free $\IZ G$-chain complexes $C_*$
for which $\Lambda(C_*)$ is $L^2$-acyclic. Recall that a finite Hilbert $\caln(G)$-chain
complex $D_*$ is \emph{$L^2$-acyclic} if for every $n \in \IZ $ its 
\emph{$n$-th   $L^2$-homology group} $H_n^{(2)}(D_*)$, which is defined to be the finite Hilbert
$\caln(G)$-module given by the quotient of the kernel of the $n$-th differential by the
closure of the image of the $(n+1)$-st differential, vanishes, or, equivalently, if for
each $n \in \IZ$ the \emph{$n$-th $L^2$-Betti number} $b_n^{(2)}(D_*)$, which is defined
to be the von Neumann dimension of $H_n^{(2)}(D_*)$, vanishes.

\begin{definition}[Weak chain contraction]
\label{def:weak_chain_contractible}
Consider a $\IZ G$-chain complex $C_*$. A \emph{weak chain contraction}
$(\gamma_*,u_*)$ for $C_*$ consists of a $\IZ G$-chain map $u_* \colon C_* \to C_*$ and 
a $\IZ G$-chain homotopy $\gamma_* \colon u_* \simeq 0_*$ such that $\Lambda(u_n)$ is a weak isomorphism
for all $n \in \IZ$ and $\gamma_n \circ u_n = u_{n+1} \circ \gamma_n $ holds for all $n \in \IZ$.  
\end{definition}

Let $C_*$ be a finite based free $\IZ G$-chain complex. Denote by $\Delta_n \colon C_n \to C_n$ 
its \emph{$n$-th combinatorial Laplace operator} which is the $\IZ G$-map 
$\Delta_n = c_{n+1} \circ c_{n+1}^* + c_n^* \circ c_n \colon C_n \to C_n$.  It has the property that
$\Lambda(\Delta_n)$ is the Laplace operator of the finite Hilbert $\caln(G)$-chain complex
$\Lambda(C_*)$, where the $n$-th Laplace operator of a finite Hilbert $\caln(G)$-chain
complex $D_*$ is the morphisms of finite Hilbert $\caln(G)$-modules $d_{n+1} \circ
d_{n+1}^* + d_n^* \circ d_n \colon D_n \to D_n$.

\begin{lemma} \label{lem:L2-acyclic_and_weak_chain_contractions}
Let $C_*$ be a finite based free $\IZ G$-chain complex. Then the following assertions are equivalent:

\begin{enumerate}[font=\normalfont]

\item \label{lem:L2-acyclic_and_weak_chain_contractions:Laplace}
$\Lambda(\Delta_n)$ is a weak isomorphism for all $n \in \IZ$;

\item \label{lem:L2-acyclic_and_weak_chain_contractions:weak_composite_zero}
There exists a weak chain contraction $(\gamma_*,u_*)$ with $\gamma_n \circ \gamma_{n-1} = 0$ for all
$n \in \IZ$;

\item \label{lem:L2-acyclic_and_weak_chain_contractions:weak}
There exists a weak chain contraction $(\gamma_*,u_*)$;

\item \label{lem:L2-acyclic_and_weak_chain_contractions:L2-acyclic}
$\Lambda(C_*)$ is $L^2$-acyclic.

\end{enumerate}

\end{lemma}
\begin{proof}~\eqref{lem:L2-acyclic_and_weak_chain_contractions:Laplace}
  $\implies$~\eqref{lem:L2-acyclic_and_weak_chain_contractions:weak_composite_zero} Define
  $\gamma_n \colon C_n \to C_{n+1}$ by $c_{n+1}^*$. Then 
  $\gamma_n \circ \gamma_{n-1} =   (c_{n+1})^* \circ c_n^* = (c_n \circ c_{n+1})^* = 0$.  
  Put $u_n = \Delta_n \colon C_n   \to C_n$. Then the collection of the $u_n$'s defines a $\IZ G$-chain map 
  $u_* \colon C_*   \to C_*$ such that $\gamma_*$ is a $\IZ G$-chain homotopy $u_*\simeq 0_*$, we have
  $\gamma_n \circ u_n = u_{n+1} \circ \gamma_n$  for $n \in \IZ$, and $\Lambda(u_n)$
  is a weak isomorphism for $n \in \IZ$.  
  \\[1mm]~\eqref{lem:L2-acyclic_and_weak_chain_contractions:weak_composite_zero} 
   $\implies$~\eqref{lem:L2-acyclic_and_weak_chain_contractions:weak} is obvious.
  \\[1mm]~\eqref{lem:L2-acyclic_and_weak_chain_contractions:weak} 
  $\implies $~\eqref{lem:L2-acyclic_and_weak_chain_contractions:L2-acyclic}
  Since $\Lambda(u_n)$ is a weak isomorphism for all $n \in \IZ$, the induced map
  $H_n^{(2)}(\Lambda(u_*)) \colon H_n^{(2)}(\Lambda(C_*)) \to H_n^{(2)}(\Lambda(C_*))$ is a weak isomorphism for all $n \in \IZ$
  by~\cite[Lemma~3.44 on page~149]{Lueck(2002)}.
  Since $\Lambda(u_*) $ is nullhomotopic, $H_n^{(2)}(\Lambda(u_*))$ is the zero map for all $n \in \IZ$. This implies that
  $H_n^{(2)}(\Lambda(C_*))$ vanishes for all $n \in \IZ$. Hence $\Lambda(C_*)$ is $L^2$-acyclic.
  \\[1mm]~\eqref{lem:L2-acyclic_and_weak_chain_contractions:L2-acyclic} 
   $\implies$~\eqref{lem:L2-acyclic_and_weak_chain_contractions:Laplace}
  This follows from~\cite[Lemma~1.13 on page~23 and Lemma~1.18 on page~24]{Lueck(2002)}
  since $\Lambda(\Delta_n)$ is the $n$-th Laplace operator of $\Lambda(C_*)$.
\end{proof}

\begin{remark}[Homomorphisms of finitely generated based free $\IZ G$-modules and $\widetilde{K}_1^w(\IZ G)$]
\label{rem:Endomorphisms_of_finitely_generated_based_free_ZG-modules_widetildeK_1_upper_w(ZG)}
Let $f \colon M \to N$ be a $\IZ G$-homomorphism of finitely generated based free $\IZ G$-modules
such that $\Lambda(f)$ is a weak isomorphism. Then we conclude from~\cite[Lemma~1.13 on page~23]{Lueck(2002)}
\[
\rk_{\IZ G}(M) = \dim_{\caln(G)}(\Lambda(M)) = \dim_{\caln(G)}(\Lambda(N)) = \rk_{\IZ G}(N).
\]
Hence we can choose a bijection between the two bases for $M$ and $N$. It induces a $\IZ G$-isomorphism
$b \colon N \to M$. Now define an element
\[
[f] := [b \circ f] \in \widetilde{K}_1^w(\IZ G).
\]
Since we are working in $\widetilde{K}_1^w(\IZ G)$, the choice of the bijection of the bases does not matter.
\end{remark}

\begin{definition}[Universal $L^2$-torsion for $L^2$-acyclic finite based free $\IZ G$-chain complexes]
\label{def:Universal_L2-torsion_for_L2-acyclic_finite_based_free_ZG-chain_complexes}
Let $C_*$ be a finite based free $\IZ G$-chain complex such that $\Lambda(C_*)$ is $L^2$-acyclic. Define its
\emph{universal $L^2$-torsion}
\[
\rho^{(2)}_u(C_*) \in \widetilde{K}^w_1(\IZ G)
\]
by
\[
\rho^{(2)}_u(C_*) = [(uc + \gamma)_{\odd}] - [u_{\odd}],
\]
where $(\gamma_*,u_*)$ is any weak chain contraction of $C_*$.
\end{definition}

We have to explain that this is well-defined.  The existence of a weak chain contraction
follows from Lemma~\ref{lem:L2-acyclic_and_weak_chain_contractions}. We have to show that
$\Lambda((uc + \gamma)_{\odd})$ and $\Lambda(u_{\odd})$ are weak isomorphisms and that the
definition is independent of the choice of  weak chain contraction.  Let $(\delta_*,v_*)$
be another weak chain contraction.  Define $\Theta_1 \colon C_{\ev} \to C_{\ev}$ by the
lower triangle matrix
\[
\Theta_1 := (v\circ u + \delta \circ \gamma)_{\ev} = 
\footnotesize{\begin{pmatrix} \ddots & \vdots & \vdots & \vdots & \ddots
\\ \ddots & vu & 0 & 0 & \ldots\\
\ldots & \delta \gamma & vu & 0 & \ldots\\
\ldots & 0 & \delta \gamma & vu & \ldots\\
\ddots & \vdots & \vdots & \ddots & \ddots 
\end{pmatrix}}
\]
Then the composite
\[
\Theta_2 \colon C_{\odd} \xrightarrow{(uc + \gamma)_{\odd}} C_{\ev}
\xrightarrow{\Theta_1} C_{\ev} \xrightarrow{(vc + \delta)_{\ev}} C_{\odd}
\]
is given by the lower triangle matrix
\[
\footnotesize{\left( \begin{array}{ccccc} \ddots & \vdots & \vdots & \vdots & \ddots
\\ \ldots & (v^2u^2)_{2n-1} & 0 & 0 & \ldots\\
\ldots & * & (v^2u^2)_{2n+1} & 0 & \ldots\\
\ldots & * & * & (v^2u^2)_{2n+3} & \ldots\\
\ddots & \vdots & \vdots & \vdots & \ddots \end{array} \right).}
\]
We conclude from~\cite[Lemma~3.37 on page~144]{Lueck(2002)} that $\Lambda(\Theta_1)$ 
and $\Lambda(\Theta_2)$ are weak isomorphisms   and we have
\[
\Theta_2 = (vc + \delta)_{\ev} \circ \Theta_1 \circ (uc + \gamma)_{\odd}.
\]
Analogously we find $\IZ G$-homomorphisms $\Theta_3$ and $\Theta_4$ such that
$\Lambda(\Theta_3)$ and $\Lambda(\Theta_4)$ are weak isomorphism and we have
\[
\Theta_4 = (vc+\delta)_{\odd} \circ \Theta_3 \circ (uc+\gamma)_{\ev}.
\]
Since we can interchange the roles of $(\gamma_*,u_*)$ and
$(\delta_*,v_*)$, we can find $\IZ G$-homo\-mor\-phisms $\Theta_i $ for $i = 5,6,7,8$ such that
$\Lambda(\Theta_i)$ is a weak isomorphism  and we have
\begin{eqnarray*}
\Theta_6 
& = & 
(uc + \gamma)_{\ev} \circ \Theta_5 \circ (vc + \delta)_{\odd};
\\
\Theta_8 
& = &
(uc+\gamma)_{\odd} \circ \Theta_7 \circ (vc+\delta)_{\ev}.
\end{eqnarray*}
We conclude that $\Lambda((uc + \gamma)_{\odd})$, $\Lambda(uc + \gamma)_{\ev})$, 
$\Lambda((vc +\delta)_{\odd})$,  and $\Lambda((vc + \delta)_{\ev})$ are weak isomorphisms. So we get well-defined
elements $[(uc + \gamma)_{\odd}]$, $[u_{\odd}]$, $[(vc + \delta)_{\ev}]$, and $[v_{\ev}]$ in $\widetilde{K}_1^w\IZ G)$.  
We have $u_{\ev} \circ (uc + \gamma)_{\odd} = (uc + \gamma)_{\odd} \circ u_{\odd}$ and
$v_{\ev} \circ (vc + \delta)_{\odd} = (vc + \delta)_{\odd} \circ v_{\odd}$. This implies
\[
{[u_{\odd}]}
 \,\, = \,\, 
{[u_{\ev}]}\quad \mbox{ and }\quad 
{[v_{\odd}]}
 \,\, = \,\,
{[v_{\ev}].}\]
Since 
\begin{eqnarray*}
{[\Theta_1] } 
& = &
 {[v_{\ev}]} + {[u_{\ev}]};
\\
{[\Theta_2]}  
& = & 
2 \cdot {[v_{\odd}]} + 2 \cdot {[u_{\odd}]};
\\
{[\Theta_2]}
& = & 
{[(vc + \delta)_{\ev}]} + {[\Theta_1]} +{[(uc + \gamma)_{\odd}]},
\end{eqnarray*}
hold in $\widetilde{K}^w_1(\IZ G)$, we get in $\widetilde{K}^w_1(\IZ G)$
\begin{eqnarray}
 [(uc + \gamma)_{\odd}] - [u_{\odd}]  
& = & 
- [(vc + \delta)_{\ev}] + [v_{\ev}].
\label{chain_contractions_in_K_1:(uc_plus_gamma}
\end{eqnarray}
Since the right hand side of the equation~\eqref{chain_contractions_in_K_1:(uc_plus_gamma}
above is independent of $(\gamma_*,u_*)$, we conclude that 
Definition~\ref{def:Universal_L2-torsion_for_L2-acyclic_finite_based_free_ZG-chain_complexes}
makes sense.

We call an exact  short sequence $0 \to M_0 \xrightarrow{i} M_1 \xrightarrow{p} M_2 \to 0$
of finitely generated based free $\IZ G$-modules \emph{based exact} if $i(B_0) \subseteq B_1$ holds and $p$ maps
$B_1 \setminus i(B_0)$ bijectively onto $B_2$ for the given $\IZ G$-basis $B_i \subseteq M_i$ for $i = 0,1,2$.
This extends in the obvious way to $\IZ G$-chain complexes.

\begin {lemma}
\label{lem:Additivity_of_universal_L2-torsion}
Let $0 \to C_* \xrightarrow{i_*}  D_* \xrightarrow{p_*} E_* \to 0$
be a based exact short sequence of finite based free  $\IZ G $-chain complexes.
Suppose that  two  of the finite Hilbert $\caln(G)$-chain complexes $\Lambda(C_*)$, $\Lambda(D_*)$  and $\Lambda(E_*)$
 are $L^2$-acyclic. Then all three are $L^2$-acyclic and we get in  $\widetilde{K}_1^w(\IZ G)$
\[
\rho^{(2)}_u(D_*) = \rho^{(2)}_u(C_*) + \rho^{(2)}_u(E_*).
\]
\end{lemma}
\begin{proof}
 All three finite Hilbert $\caln(G)$-chain complexes $\Lambda(C_*)$, $\Lambda(D_*)$  and $\Lambda(E_*)$
are $L^2$-acyclic by~\cite[Theorem~1.21 on page~27]{Lueck(2002)}.   Since $E_n$ is free
and $0 \to C_n \to D_n \to E_n$ is exact, we can choose  $\IZ G$-homomorphisms 
$r_n \colon D_n \to C_n$ and $s_n \colon E_n \to D_n$
satisfying $r_n \circ s_n = 0$, $r_n \circ i_n = \id_{C_n}$ and $p_n \circ s_n =
  \id_{E_n}$ for  each $n\in \IZ$.  Because of Lemma~\ref{lem:L2-acyclic_and_weak_chain_contractions}
 we can choose weak chain contractions $(\gamma_*,u_*)$ for $C_*$ and
  $(\epsilon_*,w_*)$ for $E_*$ such that $\gamma_{n+1}   \circ \gamma_n = 0$ and $\epsilon_{n+1} \circ \epsilon_n = 0$ 
  holds for all $n \in \IZ$.   Define
\[
\delta_n := i_{n+1} \circ  \gamma_n \circ r_n  + s_{n+1} \circ \epsilon_n \circ p_n \colon D_n \to D_{n+1}.
\]
We compute
\begin{eqnarray*}
\delta \circ \delta 
& = & 
(i \circ  \gamma \circ r  + s \circ \epsilon \circ p) \circ
(i \circ  \gamma \circ r  + s \circ \epsilon \circ p)
\\
& = &
i \circ  \gamma \circ r \circ i \circ  \gamma \circ r 
+ i \circ  \gamma \circ r \circ s \circ \epsilon \circ p
\\ & & 
\quad \quad \quad + s \circ \epsilon \circ p \circ i \circ  \gamma \circ r
+ s \circ \epsilon \circ p \circ s \circ \epsilon \circ p
\\
& = &
i \circ  \gamma \circ  \id \circ \gamma \circ r 
+ i \circ  \gamma \circ 0 \circ \epsilon \circ p
+ s \circ \epsilon \circ 0 \circ  \gamma \circ r
+ s \circ \epsilon \circ \id \circ \epsilon \circ p
\\
& = & 
i \circ  \gamma^2 \circ r 
+ s \circ \epsilon^2 \circ p
\,\, = \,\,
i \circ  0\circ r 
+ s \circ 0 \circ p
\,\,=\,\, 0.
\end{eqnarray*}

Put $v_n := d_{n+1} \circ \delta_n + \delta_{n-1} \circ d_n$.
One easily checks $v_n \circ \delta_{n-1} = \delta_{n-1} \circ v_{n-1}$ for every $n \in \IZ$.
The following diagrams commute

\[
\xymatrix@R0.5cm{
0 \ar[r] 
&
C_n \ar[r]^{i_n} \ar[d]^{\gamma_n}
& 
D_n \ar[r]^{p_n} \ar[d]^{\delta_n}
& 
E_n \ar[r] \ar[d]^{\epsilon_n}
& 
0
\\
0 \ar[r] 
&
C_{n+1} \ar[r]^{i_{n+1}}
& 
D_{n+1} \ar[r]^{p_{n+1}}
& 
E_{n+1} \ar[r]
& 
0
}
\]

\[
\xymatrix@R0.5cm{
0 \ar[r] 
&
C_n \ar[r]^{i_n} \ar[d]^{u_n}
& 
D_n \ar[r]^{p_n} \ar[d]^{v_n}
& 
E_n \ar[r] \ar[d]^{w_n}
& 
0
\\
0 \ar[r] 
&
C_n \ar[r]^{i_n}
& 
D_n \ar[r]^{p_n}
& 
E_n \ar[r]
& 
0
}
\]
and
\[
\xymatrix@R0.75cm@C1.2cm{
0 \ar[r] 
&
C_{\odd} \ar[r]^{i_{\odd}} \ar[d]^{(uc + \gamma)_{\odd}}
& 
D_{\odd} \ar[r]^{p_{\odd}} \ar[d]^{(vd + \delta)_{\odd}}
& 
E_{\odd} \ar[r] \ar[d]^{(we + \epsilon)_{\odd}}
& 
0
\\
0 \ar[r] 
&
C_{\odd} \ar[r]^{i_{\odd}}
& 
D_{\odd} \ar[r]^{p_{\odd}}
& 
E_{\odd} \ar[r]
& 
0.
}
\]
Since $\Lambda(u_n)$ and $\Lambda(w_n)$ are weak isomorphisms, the same is true for $\Lambda(v_n)$
by~\cite[Lemma~3.37 on page~144]{Lueck(2002)}.
Hence $(\delta_*,v_*)$ is a weak chain contraction for $D_*$.
Since each of the short exact sequences $0 \to C_n \xrightarrow{i_n} D_n \xrightarrow{p_n} E_n \to 0$
is based exact by assumption, we get in $\widetilde{K}^w_1(\IZ G)$
\begin{eqnarray*}
[(vd + \delta)_{\odd}] 
& = & 
[(uc + \gamma)_{\odd}] + [(we + \epsilon)_{\odd}]; 
\\
{[v_{\odd}]}
& = & 
[u_{\odd}] + [w_{\odd}];
\\
 \rho^{(2)}_u(D_*) 
& = & 
\rho^{(2)}_u(C_*) + \rho^{(2)}_u(E_*).
\end{eqnarray*}
This finishes the proof of Lemma~\ref{lem:Additivity_of_universal_L2-torsion}.
\end{proof}

 Let $f \colon(C_*,c_*) \to (D_*,d_*)$ be a $\IZ G$-chain map of 
  $\IZ G$-chain complexes. We denote by $\cone(f_*)$  the \emph{cone of $f$}, this is the chain complex  which is given by the modules $C_{n+1}\oplus D_n$ and where the differential from $C_{n}\oplus D_{n+1}\to C_{n-1}\oplus D_{n}$ is given by
 \[ \begin{pmatrix} -c_{n} &0 \\ f_{n}& d_{n+1}\end{pmatrix}.\]
 Furthermore the \emph{supension} of a chain complex $(C_*,c_*)$ is defined as the  chain complex where the $n$-th boundary map is given by $C_{n-1}\xrightarrow{-c_{n-1}}C_{n-2}$.

\begin{lemma} \label{lem:universal_L2-torsion_and_chain_homotopy_equivalences}
  Let $f \colon C_* \to D_*$ be a $\IZ G$-chain homotopy equivalence of finite based free
  $\IZ G$-chain complexes.  Denote by $\rho(\cone(f_*)) \in \widetilde{K}_1(\IZ G)$ the classical
  torsion of the finite based free contractible $\IZ G$-chain complex $\cone(f_*)$.  Suppose
  that $\Lambda(C_*)$ or $\Lambda(D_*)$ is $L^2$-acyclic.
  Then both $\Lambda(C_*)$ and $\Lambda(D_*)$ are $L^2$-acyclic, and we get in
  $\widetilde{K}^w_1(\IZ G)$
  \[
  \rho^{(2)}_u(D_*) - \rho^{(2)}_u(C_*) = \zeta(\rho(\cone(f_*)))
  \]
  for  the canonical homomorphism $\zeta \colon \widetilde{K}_1(\IZ G) \to \widetilde{K}^w_1(\IZ G)$.
\end{lemma}
\begin{proof} Since $f_*$ is a $\IZ G$-chain homotopy equivalence, $\cone(f_*)$ is a
  finite based free contractible $\IZ G$-chain complex. In particular $\Lambda(\cone(f_*))$
  is $L^2$-acyclic. One easily checks that 
  $\zeta \colon \widetilde{K}_1(\IZ G) \to   \widetilde{K}^w_1(\IZ G)$ 
  maps $\rho(\cone(f_*))$ to $\rho^{(2)}_u(\cone(f_*))$ since a
  chain contraction $\gamma_*$ for $\cone(f_*)$ defines a weak chain contraction
  $(\gamma_*,\id_{\cone(f_*)})$ for $\cone(f_*)$. Now apply
  Lemma~\ref{lem:Additivity_of_universal_L2-torsion} to the obvious based exact short
  sequences of finite based free $\IZ G$-chain complexes 
  $0 \to D_* \to \cone(f_*)   \to \Sigma C_* \to 0$ and use the obvious fact that $\Lambda(\Sigma C_*)$ is
  $L^2$-acyclic, if and only if $\Lambda(C_*)$ is $L^2$-acyclic, and in this case 
  $\rho^{(2)}_u(\Sigma     C_*) = -\rho^{(2)}_u(C_*)$ holds.
  \end{proof}

\subsection{The $K$-group $K_1^{w, \ch}(\IZ G)$}
\label{subsec:K_1_upper_w,ch(ZG)}

\begin{definition}[$\widetilde{K}_1^{w,\ch}(\IZ G)$ and $\widetilde{K}_1^{\ch}(\IZ G)$ ]
  \label{def:widetilde(K)_1_upper_w_ch_(omega)_and_omega-torsion}
  Let $\widetilde{K}_1^{w, \ch }(\IZ G)$ be the abelian group defined in terms of
  generators and relations as follows.  Generators $[C_*]$ are given by 
  (basis preserving isomorphism classes of)  finite based free 
  $\IZ G$-chain complexes $C_*$ such that $\Lambda(C_*)$ is $L^2$-acyclic.   Whenever we have
  a short based exact sequence of finite based free   $\IZ G$-chain
  complexes $0 \to C_* \to D_* \to E_* \to 0$ such that two (and hence all) of the finite Hilbert $\caln(G)$-chain 
  complexes $\Lambda(C_*)$, $\Lambda(D_*)$, and $\Lambda(E_*)$ are $L^2$-acyclic, we require the relation
  \[ 
  [D_*] = [C_*] + [E_*].
 \]
  We also require that for any $n$ we have
  \[
  \big[0\, \to\,\IZ G\,\xrightarrow{\id_{\IZ G}}\,  \IZ G\, \to \,0\big] \,=\,  \big[0\, \to\,\IZ G\,\xrightarrow{-\id_{\IZ G}}\,  \IZ G\, \to \,0\big]\, = \,0,
  \]
  where the two non-zero terms lie in dimension $n$ and $n+1$.
  Let $\widetilde{K}_1^{\ch }(\IZ G)$ be the abelian group defined analogously, where we
   replace everywhere the condition that $\Lambda(C_*)$ is $L^2$-acyclic by the stronger
  condition that $C_*$ is contractible as $\IZ G$-chain complex.
  \end{definition}

\begin{theorem}[Chain complexes versus automorphisms]
\label{the:chain_complexes_versus_automorphisms}
The classical torsion and the universal $L^2$-torsion induce isomorphisms such that
the following diagram commutes, where the horizontal maps  are the obvious forgetful homomorphisms.
\[
\xymatrix@R0.5cm{\widetilde{K}_1^{\ch}(\IZ G) \ar[r]^-{\zeta} \ar[d]_{\rho}^{\cong}
& 
\widetilde{K}_1^{w, \ch}(\IZ G) \ar[d]^{\rho^{(2)}_u}_{\cong}
\\
\widetilde{K}_1(\IZ G) \ar[r]_{\zeta} 
& 
\widetilde{K}_1^w(\IZ G).
}
\]
The inverses of $\rho$ and $\rho^{(2)}_u$ are given by the obvious maps regarding a
homomorphism of finitely generated based free $\IZ G$-modules as a $1$-dimensional finite
based free $\IZ G$-chain complex.
\end{theorem}

\begin{proof}
  The map $\rho^{(2)}_u \colon \widetilde{K}_1^{w, \ch}(\IZ G) \to \widetilde{K}_1^w(\IZ
  G)$ is well-defined by Lemma~\ref{lem:Additivity_of_universal_L2-torsion}.  One easily
  checks that the diagram appearing in
  Theorem~\ref{the:chain_complexes_versus_automorphisms} commutes. It remains to show that
  the two vertical maps are isomorphisms. We start out with the left vertical map.

Next we define a homomorphism
\begin{eqnarray}
\el \colon \widetilde{K}_1^w(\IZ G) & \xrightarrow{\cong} & \widetilde{K}_1^{w,\ch}(\IZ G).
\label{el:K_1_to_K_s_ch}
\end{eqnarray} 
Consider a $\IZ G$-endomorphism $f \colon \IZ G^n \to \IZ G^n$ 
such that $\Lambda(f)$ is a weak isomorphism.  We want to define
\begin{eqnarray}
\el([f]) := [\el(f \colon \IZ G^n \to \IZ G^n)].
\label{homo_el}
\end{eqnarray}
where $\el(f \colon \IZ G^n \to \IZ G^n)$ is the $1$-dimensional finite based free 
$\IZ G$-chain complex whose first differential is $f$. Since $\Lambda(f)$ is a weak
isomorphism, $\Lambda\bigl(\el(f \colon \IZ G^n \to \IZ G^n)\bigr)$ is $L^2$-acyclic and
defines an element in $[\el(f \colon\IZ G^n \to \IZ G^n)]$ in $K_1^{w,\ch}(\IZ G)$. We
have to check that the relations in $K_1^w(\IZ G)$ are satisfied.  This follows for the
additivity relation directly from Lemma~\ref{lem:Additivity_of_universal_L2-torsion}.  It
remains to show for $\IZ G$-endomorphisms $f_1,f_2 \colon \IZ G^n \to \IZ G^n$ such that
$\Lambda(f_1)$ and $\Lambda(f_2)$ are weak isomorphisms that we get in
$\widetilde{K}_1^{w,\ch}( \IZ G)$
\begin{eqnarray*}
{[\el(f_2 \circ f_1)]} 
&= & 
{[\el(f_1)]} + {[\el(f_2)]}.
\end{eqnarray*}
Consider the chain map
\[
h_* \colon \Sigma^{-1} \el(f_2) \to  \el(f_1)
\]
given by $h_1 = \id_{\IZ G^n}$ and $h_i = 0$ for $i \not=1$. We conclude from
the short based exact sequence of finite based free $\IZ G$-chain complexes 
$0 \to \el(f_1) \to \cone(h_*) \to \el(f_2) \to 0$ 
that we get in 
$\widetilde{K}_1^{w, \ch}(\IZ G)$
\[
[\cone(h_*)] = [\el(f_1)] +[\el(f_2)].
\]
There is also a based exact short sequence of finite based free
$\IZ G$-chain complexes 
$0 \to \el(f_2 \circ f_1) \xrightarrow{j} \cone(h_*) \to \el(\id_{\IZ G^n}) \to 0$, 
given by
\[
\xymatrix@!C= 3em{
0 \ar[r]
&
\IZ G^n \ar[rr]^-{\footnotesize{\begin{pmatrix} f_1 \\ - \id_{\IZ G^n} \end{pmatrix}}} \ar[d]^{f_2 \circ f_1} 
&&
\IZ G^n \oplus \IZ G^n \ar[rr]^-{\footnotesize{\begin{pmatrix} \id_{\IZ G^n} & f_1 \end{pmatrix}}} 
\ar[d]^{\footnotesize{\begin{pmatrix} f_2  & 0 \\ \id_{\IZ G^n} & f_1 \end{pmatrix}}}
&&
\IZ G^n \ar[r] \ar[d]^{\id_{\IZ G^n}} 
& 0
\\
0 \ar[r]
&
\IZ G^n \ar[rr]_-{\footnotesize{\begin{pmatrix} \id_{\IZ G^n} \\ 0 \end{pmatrix}}}
 &&
\IZ G^n \oplus \IZ G^n \ar[rr]_-{\footnotesize{\begin{pmatrix} 0 & \id_{\IZ G^n} \end{pmatrix}}} 
&& \IZ G^n \ar[r] 
&
0.
}
\]
It yields the equation in $\widetilde{K}_1^{w,\ch}(\IZ G)$
\[
[\cone(h_*)] =  [\el(f_2 \circ f_1)] + [\el(\id_{\IZ G^n})].
\]
This implies
\begin{eqnarray}
\hspace{0.65cm}[\el(f_2 \circ f_1)] 
& = & 
[\cone(h_*)] -  [\el(\id_{\IZ G^n})]
\label{el:K_1_toK_s_ch_composition}
\\
& = & 
[\el(f_1] +[\el(f_2)] - [\el(\id_{\IZ G^n})]
\nonumber
\\
& = & 
[\el(f_1)] +[\el(f_2)] - n \cdot [\el(\id_{\IZ G})]
\nonumber
\,\, = \,\,
[\el(f_1)] +[\el(f_2)].
\nonumber 
 \end{eqnarray}
This finishes the proof that the homomorphism $\el$ announced in~\eqref{homo_el} is well-defined.

One easily checks that $\rho^{(2)}_u \circ \el \colon K_1^w(\IZ G) \to K_1^w(\IZ G)$ is the identity. 
In order to  show that the homomorphisms $\rho^{(2)}_u$ and $\el$ are bijective and inverse to one another,
it remains to show that $\el \colon K_1^w(\IZ G) \to K_1^{w, \ch} (\IZ G)$ is surjective.

We have to show for any finite based free $ \IZ G$-chain complex $C_*$ with the property
that $\Lambda(C_*)$ is $L^2$-acyclic that $[C_*]$ lies in the image of $\el$. By possibly
suspending $C_*$, we can assume without loss of generality that $C_n = 0$ for $n \le
-1$. Now we do induction over the dimension $d$ of $C_*$.  The induction beginning $d \le
1$ is obvious, the induction step from $d-1\ge 1 $ to $d$ is done as follows.

Choose a weak chain contraction  $\gamma_*$ for $C_*$. Define a $\IZ G$-chain map
\[
i_* \colon \el(c_1 \circ \gamma_0 \colon C_0 \to C_0) \to C_*
\]
by $i_1 = \gamma_0$, $i_0 = \id_{C_0}$ and $i_k = 0$ for $k \not=0,1$.  We conclude from
Lemma~\ref{lem:Additivity_of_universal_L2-torsion} applied to the short
based exact sequence 
$0 \to C_* \to \cone (i_*) \to \Sigma \el(c_1 \circ \gamma_0 \colon C_0 \to C_0 ) \to 0$ 
that $\cone (i_*)$ is a finite based free $\IZ G$-chain complex such that $\Lambda(\cone(f_*))$ is $L^2$-acyclic  
and we get in $\widetilde{K}_1^{w,\ch}(\IZ G)$
\[
[\cone(i_*)] = [(C_*)] -  \bigl[\el(c_1 \circ \gamma_0 \colon  C_0 \to C_0 )\bigr].
\]
Define a $\IZ G$-chain map 
\[
j_* \colon \el(\id_{C_0}) \to \cone(i_*)
\]
by $j_0 = \id_{C_0} \colon C_0 \to C_0$, $j_1 = \footnotesize{\begin{pmatrix}\id_{C_0} 
\\  0 \end{pmatrix}} \colon C_0 \to C_0 \oplus C_1$.  
Then we can equip $\coker(j_*)$ with
an obvious $\IZ G$-basis such that we obtain a  based exact short sequence of finite based free
$\IZ G$-chain complexes $0 \to \el(\id_{C_0}) \xrightarrow{j_*} \cone(i_*) \to \coker(j_*) \to 0$.  
We conclude from Lemma~\ref{lem:Additivity_of_universal_L2-torsion} 
that $\coker(j_*)$ is $L^2$-acyclic  and we get in $\widetilde{K}_1^{w,\ch}(\IZ G )$
\[
[\cone(i_*)] = [\el(\id_{C_0})] + [\coker(j_*)].
\]
Since $[\el(\id_{C_0})]$ and $\bigl[\el(c_1 \circ \gamma_0 \colon C_0 \to C_0 )\bigr]$ lie in
the image of $\el$, it suffices to show that $[\coker(j_*)]$ lies in the image of
$\el$. Since $\coker(j_*)$ has dimension $\le d$ and its zeroth-chain module is trivial,
this follows from the induction hypothesis. This finishes the proof that
the homomorphisms $\rho^{(2)}_u$ and $\el$ are bijective and inverse to one another.

The proofs for  $\rho \colon \widetilde{K}_1^{\ch}(\IZ G)  \to \widetilde{K}_1(\IZ G)$ are analogous.
\end{proof}

  \begin{remark}[Universal property of the universal $L^2$-torsion]
    \label{rem:universal_property_of_the_universal_L2-torsion}
    An \emph{additive $L^2$-torsion invariant} $(A,a)$ consists of an abelian group $A$
    and an assignment which associates to a finite based free $\IZ G$-chain complex $C_*$
    such that $\Lambda(C_*)$ is $L^2$-acyclic, an element $a(C_* )\in A$ such that for any
    based exact short sequence of such $\IZ G$-chain complexes $0 \to C_* \to D_* \to E_*
    \to 0$ we get
    \[
    a(D_*) = a(C_*) + a(E_*),
    \]
    and we have
    \[
    a(\el(\pm \id_{\IZ G})) = 0.
    \]
    We call an additive $L^2$-torsion invariant  $ (U,u)$ \emph{universal} if for every
    additive $L^2$-torsion invariant $(A,a)$ there is precisely one group homomorphism 
   $f  \colon U \to A$ such that for every finite based free $\IZ G $-chain
    complex $C_*$ for which $\Lambda(C_*)$ is $L^2$-acyclic, we have $f(u(C_*)) = a(C_*)$.
    
   Theorem~\ref{the:chain_complexes_versus_automorphisms}  implies 
   that $(\tilde{K}_1^w(\IZ G), \rho^{(2)}_u )$ is the universal additive $L^2$-torsion invariant.
  \end{remark}

  We obtain an involution $\ast \colon \widetilde{K}_1^w(\IZ G) \to \widetilde{K}_1^w(\IZ G)$ by sending
  $[r_A\colon \IZ G^n \to \IZ G^n]$ to $[r_{A^*} \colon \IZ G^n \to \IZ G^n]$.

  \begin{lemma} \label{lem:rho(2)_u(C_aat)_in_terms_of_the_combinatorial_Laplace_operator}
   
    Let $C_*$ be a finite based free $\IZ G$-chain complex such that $\Lambda(C_*)$ is
    $L^2$-acyclic.  Then for $n \in \IZ$ the combinatorial Laplace operator is a $\IZ G$-map
    $\Delta_n \colon C_n \to C_n$ such that $\Lambda(\Delta_n)$ is a weak isomorphism, 
    and we get in $\widetilde{K}_1^w(\IZ G)$
    \[
    \rho^{(2)}_u(C_*)  + \ast(\rho^{(2)}_u(C_*)) =   - \sum_{n \ge 0} (-1)^n \cdot n \cdot [\Delta_n].
    \]
  \end{lemma}
  \begin{proof} 
    Lemma~\ref{lem:L2-acyclic_and_weak_chain_contractions} shows that $\Lambda(\Delta_n)$
    is a weak isomorphism for all $n \in \IZ$.  A finite based free $\IZ G$-chain complex
    has length $\le l$ if there exist natural numbers $n_-$ and $n_+$ with $n_- \le n_+$
    such that $C_n \not= 0$ implies $n_- \le n \le n_+$ and $l = n_+ -n_-$. We prove
    Lemma~\ref{lem:rho(2)_u(C_aat)_in_terms_of_the_combinatorial_Laplace_operator} by
    induction over $l$. The induction beginning $l= 1$ is done as follows.  Choose a
    natural number $n_+$ such that $C_*$ is concentrated in dimensions $n_+$ and $n_+ -1$.
    Since $\Delta_{n_+-1} \circ c_{n_+} = c_{n_+} \circ \Delta_{n_+}$ holds, we get in
    $\widetilde{K}_1^w(\IZ G)$
\[
[\Delta_{n_+-1}] = [\Delta_{n_+-1}\circ c_{n_+}] -  [c_{n_+}] = [c_{n_+} \circ \Delta_{n_+}] -  [c_{n_+}] =  [\Delta_{n_+}].
\] 
Now we compute
\begin{eqnarray*}
\rho^{(2)}_u(C_*) + \ast(\rho^{(2)}_u(C_*))
& = & 
(-1)^{n_++1} \cdot {[c_{n_+}]} + \ast\bigl((-1)^{n_++1} \cdot {[c_{n_+}]}\bigr) 
\\
& = & 
(-1)^{n_++1} \cdot {[c_{n_+}]} + (-1)^{n_++1} \cdot {[c_{n_+}^*]}
\\
& = & 
- (-1)^{n_+} \cdot {[c_{n_+}^* \circ c_{n_+}]} 
\\
 & = & 
- (-1)^{n_+} \cdot {[\Delta_{n_+}]}
\\
& = & 
- \left( (-1)^{n_+} \cdot n_+ \cdot {[\Delta_{n_+}]} + (-1)^{n_+-1} \cdot (n_+-1) \cdot {[\Delta_{n_+}]}\right)
\\
& = & 
- \left((-1)^{n_+} \cdot n_+ \cdot {[\Delta_{n_+}]} + (-1)^{n_+-1} \cdot (n_+-1) \cdot {[\Delta_{n_+-1}]}\right)
\\
& = & 
- \sum_{n \in \IZ} (-1)^n \cdot n \cdot {[\Delta_n]}.
\end{eqnarray*}
The induction beginning from $l -1 \ge 1$ to $l$ is done as follows.  Choose integers
$n_-$ and $n_+$ with $n_- \le n_+$ such that $C_n \not= 0$ implies $n_- \le n \le n_+$ and
$n_+ -n_- \le l$. Let $\el_{n_-}(\Delta_{n_-})_*$ be the finite based free $\IZ G$-chain
complex concentrated in dimensions $n_-+1$ and $n_-$ whose $(n_-+1)$-st differential is
$\Delta_{n_-} \colon C_{n_-} \to C_{n_-}$.  Define a $\IZ G$-chain map 
$f_* \colon \el_{n_-}(\Delta_{n_-})_*\to C_*$ by putting $f_{n_-} = \id_{C_{n_-}}$ and $f_{n_-+1} =
c_{n_-+1}^*$. Let $\cone(f_*)$ be its mapping cone. Then we obtain a based exact short
sequence of finite based free $\IZ G$-chain complexes 
$0 \to C_*\to \cone(f_*) \to \Sigma \el_{n_-}(\Delta_{n_-})_*\to 0$. Since $\Lambda(\el_{n_-}(\Delta_{n_-})_*)$ 
and $\Lambda(C_*)$ are $L^2$-acyclic, Lemma~\ref{lem:Additivity_of_universal_L2-torsion}
implies that also $\cone(f_*)$ is $L^2$-acyclic and we get in $\widetilde{K}_1^w(\IZ G)$
\begin{eqnarray}
\rho^{(2)}_u(C_*)  
& = &
\rho^{(2)}_u(\cone(f_*) ) - \rho^{(2)}_u\bigl(\Sigma  \el_{n_-}(\Delta_{n_-})_*\bigr).
\label{lem:rho(2)_u(C_aat)_in_terms_of_the_combinatorial_Laplace_operator:(1)}
\end{eqnarray}
Let $\el_{n_-}(\id_{C_{n_-}})_*$ be the finite based free $\IZ G$-chain complex
concentrated in dimensions $n_-+1$ and $n_-$ whose $(n_-+1)$-st differential is $\id \colon
C_{n_-} \to C_{n_-}$. Let $i_* \colon \el_{n_-}(\id_{C_{n_-}})_*\to \cone(f_*)_* $ be
the $\IZ G$-chain map which is given by the identity $\id_{C_{n_-}}$ in degree $n_-$ and
by the obvious inclusion $C_{n_-} \to C_{n_-} \oplus C_{n_-+1}$ in degree $n_- +1$. The
cokernel of $i_*$  is the finite based free $\IZ G$-chain complex $D_*$ which is
concentrated in dimensions $n$ for $n_- + 1 \le n \le n_+$ and given by
\begin{multline*} \cdots \to 0 \to C_{n_+} \xrightarrow{c_{n_+}} C_{n_+-1} 
\xrightarrow{c_{n_+-1 }}  \cdots \xrightarrow{c_{n_-+4}} C_{n_- +3}  
\xrightarrow{\begin{pmatrix} 0 & c_{n_- +3}  \end{pmatrix}} C_{n_-}   \oplus C_{n_- +2}  
\\
\xrightarrow{\begin{pmatrix} - c_{n_-+1}^* & c_{n_-+2} \end{pmatrix}} C_{n_- + 1} \to 0 \to \cdots
\end{multline*}
We obtain a based exact short sequence of finite based free $\IZ G$-chain complexes 
$0 \to  \el_{n_-}(\id_{C_{n_-}})_*  \xrightarrow{i_*} \cone(f_*) \to D_* \to 0$. 
Lemma~\ref{lem:Additivity_of_universal_L2-torsion} implies that $\Lambda(D_*)$ is 
$L^2$-acyclic and we get in $\widetilde{K}_1^w(\IZ G)$
\begin{equation}
\rho^{(2)}_u(\cone(f_*))  =  \rho^{(2)}_u\bigl(\el_{n_-}(\id_{C_{n_-}})_*\bigr)  + \rho^{(2)}_u(D_*) =  \rho^{(2)}_u(D_*).
\label{lem:rho(2)_u(C_aat)_in_terms_of_the_combinatorial_Laplace_operator:(2)}
\end{equation}
Combining~\eqref{lem:rho(2)_u(C_aat)_in_terms_of_the_combinatorial_Laplace_operator:(1)}
and~\eqref{lem:rho(2)_u(C_aat)_in_terms_of_the_combinatorial_Laplace_operator:(2)} yields
\begin{eqnarray}
\rho^{(2)}_u(C_*) 
& = & 
\rho^{(2)}_u(D_*) - \rho^{(2)}_u\bigl(\Sigma  \el_{n_-}(\Delta_{n_-})_*\bigr).
\label{lem:rho(2)_u(C_aat)_in_terms_of_the_combinatorial_Laplace_operator:(3)}
\end{eqnarray}
The induction hypothesis applies to $D_*$ since its length is $\le l-1$. Hence we get, 
if $\Delta_n'$ is the combinatorial Laplace operator of $D_*$, that
\[ \begin{array}{rcl}
&&\rho^{(2)}_u(D_*) + \ast\bigl(\rho^{(2)}_u(D_*) \bigr)
\\
& = &
- \sum_{n \in \IZ} (-1)^n \cdot n \cdot  {[\Delta'_n]}
\\
& = & 
- (-1)^{n_-+1} \cdot (n _- +1) \cdot {[\Delta_{n_-+1}]} 
- (-1)^{n_-+2} \cdot (n _- +2) \cdot {[\Delta_{n_-} \oplus \Delta_{n_-+2}]}
\\
& & \hspace{80mm} - \smsum{n = n_-+3}{n_+}  (-1)^n \cdot n \cdot  {[\Delta_n]}
\\
& = & 
- \smsum{n = n_-+1}{n_+} (-1)^n \cdot n \cdot  {[\Delta_n]} -  (-1)^{n_-+2} \cdot (n _- +2) \cdot {[\Delta_{n_-}]}
\\
& = & 
- \smsum{n = n_-}{n_+}  (-1)^n \cdot n \cdot  {[\Delta_n]} -  (1)^{n_-+2} \cdot 2 \cdot {[\Delta_{n_-}]}\\
&=&
- \smsum{n \in \IZ}{}  (-1)^n \cdot n \cdot  {[\Delta_n]} +  (1)^{n_-+1} \cdot 2 \cdot {[\Delta_{n_-}]}.
\end{array}\]
We compute
\[ 
\begin{array}{rlc}
& \,\, 
\rho^{(2)}_u\bigl(\Sigma  \el_{n_-}(\Delta_{n_-})_*\bigr) 
+ \ast\bigl(\rho^{(2)}_u\bigl(\Sigma  \el_{n_-}(\Delta_{n_-})_*\bigr)\bigr)
\\
& = \,\,
(-1)^{n_- + 1} \cdot {[-\Delta_{n_-}]} + \ast\bigl((-1)^{n_- + 1}  {[-\Delta_{n_-}]}\bigr)
\,\,= \,\,
(-1)^{n_- + 1} \cdot 2 \cdot  {[\Delta_{n_-}]}.
\end{array}\]
Now Lemma~\ref{lem:rho(2)_u(C_aat)_in_terms_of_the_combinatorial_Laplace_operator}
 follows from~\eqref{lem:rho(2)_u(C_aat)_in_terms_of_the_combinatorial_Laplace_operator:(3)} 
together with the last two equalities.
\end{proof}


\subsection{Review of division and rational closure}
\label{subsec:Review_of_division_closure}

Let $R$ be a subring of the ring $S$. (Here and throughout the paper a ring is understood to be an associative ring with 1, which is not necessarily commutative.) The \emph{division closure} $\cald(R \subseteq S) \subseteq S$ 
is the smallest subring of $S$ which contains $R$ and is division closed,
i.e., any element $x \in \cald(R \subseteq S)$ which is invertible in $S$ is already invertible in
$\cald(R \subseteq S)$. The \emph{rational closure} $\calr(R \subseteq S) \subseteq S$ is
the smallest subring of $S$ which contains $R$ and is rationally closed, i.e., for any
natural number $n$ and matrix $A \in M_{n,n}(\calr(R \subseteq S))$ which is invertible in
$S$ is already invertible over $\calr(R \subseteq S)$. The division closure and the
rational closure always exist. Obviously 
$R \subseteq \cald(R \subseteq S) \subseteq \calr(R \subseteq S) \subseteq S$.

Consider a group $G$. Let $\caln(G)$ be the group von Neumann algebra which can be
identified with the algebra $\calb(L^2(G),L^2(G))^G$ of bounded $G$-equivariant operators
$L^2(G) \to L^2(G)$.  Denote by $\calu(G)$ the algebra of operators which are affiliated
to the group von Neumann algebra, see \cite[Section~8]{Lueck(2002)} for details.   
This is the same as the Ore localization of $\caln(G)$
with respect to the multiplicatively closed subset of non-zero divisors in $\caln(G)$,
see~\cite[Theorem~8.22~(1)]{Lueck(2002)}. By the right regular representation we can embed 
$\IC G$ and hence also $\IZ G$ as subring in $\caln(G)$.  We will denote by $\calr(G)$ and
$\cald(G)$ the division and the rational closure of $\IZ G$ in $\calu(G)$.  So we get a
commutative diagram of inclusions of rings
\[
\xymatrix@!C= 8em@R0.5cm{
\IZ G \ar[r] \ar[d]
&
\caln(G) \ar[dd]
\\
\cald(G) \ar[d]
&
\\
\calr(G) \ar[r]
& \calu(G)}
\]

\begin{lemma} \label{lem:L2-acyclic_and_ZG_subseteq_calr(G)-contractible}
Let $C_*$ be a finite based free $\IZ G$-chain complex.  Then the following 
assertions are equivalent:

\begin{enumerate}[font=\normalfont]

\item \label{lem:L2-acyclic_and_ZG_subseteq_calr(G)-contractible:L2-acyclic}
$\Lambda(C_*)$  is $L^2$-acyclic;

\item \label{lem:L2-acyclic_and_ZG_subseteq_calr(G)-contractible:Laplace_weak}
The operator $\Lambda(\Delta_n) \colon \Lambda(C_n) \to \Lambda(C_n)$ 
is a weak isomorphism for all $n \in \IZ$;

\item \label{lem:L2-acyclic_and_ZG_subseteq_calr(G)-contractible:Laplace_U(G)}
The $\calu(G)$-homomorphism $\id_{\calu(G)} \otimes_{\IZ G} \Delta_n \colon \calu(G) \otimes_{\IZ G} C_n 
\to \calu(G) \otimes_{\IZ G} C_n$
is an isomorphism for all $n \in \IZ;$

\item \label{lem:L2-acyclic_and_ZG_subseteq_calr(G)-contractible:Laplace_R(G)}

The $\calr(G)$-homomorphism $\id_{\calr(G)} \otimes_{\IZ G} \Delta_n \colon \calr(G) \otimes_{\IZ G} C_n 
\to \calr(G) \otimes_{\IZ G} C_n$
is an isomorphism for all $n \in \IZ$;

\item \label{lem:L2-acyclic_and_ZG_subseteq_calr(G)-contractible:contractible_over_calu}
The $\calu(G)$-chain complex $\calu(G) \otimes_{\IZ G} C_*$ is contractible; 

\item \label{lem:L2-acyclic_and_ZG_subseteq_calr(G)-contractible:contractible_over_calr}
The $\calr(G)$-chain complex $\calr(G) \otimes_{\IZ G} C_*$ is contractible.

\end{enumerate}
\end{lemma}
\begin{proof}~\eqref{lem:L2-acyclic_and_ZG_subseteq_calr(G)-contractible:L2-acyclic}
$\Longleftrightarrow$~\eqref{lem:L2-acyclic_and_ZG_subseteq_calr(G)-contractible:Laplace_weak}
This has already been proved in Lemma~\ref{lem:L2-acyclic_and_weak_chain_contractions}.
\\[1mm]~\eqref{lem:L2-acyclic_and_ZG_subseteq_calr(G)-contractible:Laplace_weak}
$\Longleftrightarrow$~\eqref{lem:L2-acyclic_and_ZG_subseteq_calr(G)-contractible:Laplace_U(G)}
This follows from~\cite[Theorem~6.24 on page~249 and Theorem~8.22~(5) on page~327]{Lueck(2002)}.
\\[1mm]~\eqref{lem:L2-acyclic_and_ZG_subseteq_calr(G)-contractible:Laplace_U(G)}
$\Longleftrightarrow$~\eqref{lem:L2-acyclic_and_ZG_subseteq_calr(G)-contractible:Laplace_R(G)}
This follows from the definition of the rational closure.
\\[1mm]~\eqref{lem:L2-acyclic_and_ZG_subseteq_calr(G)-contractible:Laplace_R(G)} 
$\implies$~\eqref{lem:L2-acyclic_and_ZG_subseteq_calr(G)-contractible:contractible_over_calr}
The collection of the $\IZ G$-maps $c_{n+1}^* \colon C_n \to C_{n+1}$ 
defines $\IZ G$-chain homotopy $\Delta_* \simeq 0_*$,
where $\Delta_* \colon C_* \to C_*$ is the $\IZ G$-chain map given by $\Delta_n$ in degree $n$. 
Therefore we get a $\calr(G)$-chain isomorphism 
$\id_{\calr(G)} \otimes_{\IZ G} \Delta_*\colon \calr(G) \otimes_{\IZ G} C_* \to \calr(G) \otimes_{\IZ G} C_*$
which is $\calr(G)$-nullhomotopic. Hence $\calr(G) \otimes_{\IZ G} C_*$ is contractible.
\\[1mm]~\eqref{lem:L2-acyclic_and_ZG_subseteq_calr(G)-contractible:contractible_over_calr}
$\implies$~\eqref{lem:L2-acyclic_and_ZG_subseteq_calr(G)-contractible:contractible_over_calu}
This is obvious.
\\[1mm]~\eqref{lem:L2-acyclic_and_ZG_subseteq_calr(G)-contractible:contractible_over_calu}
$\implies$~\eqref{lem:L2-acyclic_and_ZG_subseteq_calr(G)-contractible:L2-acyclic}
We conclude from~\cite[Theorem~6.24 on page~249 and Theorem~8.29~(5) on page~330]{Lueck(2002)}
that $b_n^{(2)}(\Lambda(C_*)) = 0$ for all $n \in \IZ$.
This finishes the proof of Lemma~\ref{lem:L2-acyclic_and_ZG_subseteq_calr(G)-contractible}.
\end{proof}

With the notation we just introduced we can now formulate the following proposition of Linnell--L\"uck, see \cite[Theorem~0.1]{Linnell-Lueck(2016)}.
For many torsionfree groups it gives an alternative description of $K_1^w(\IZ [G])$.

\begin{proposition}
 Let $\calc$ be the
  smallest class of groups which contains all free groups and is closed under directed
  unions and extensions with elementary amenable quotients.  Let $G$ be a torsionfree
  group which belongs to $\calc$. Then the following hold:
 \begin{enumerate}[font=\normalfont]
 \item  $K_1^w(\IZ [G])$ is isomorphic to
  $K_1(\cald(G))$. 
  \item The ring  $\cald(G)$ is a skew field.
  \item The group 
  $K_1(\cald(G))$ is the abelianization of the multiplicative group of units in
  $\cald(G)$.
  \end{enumerate}
  \end{proposition}


 \typeout{------------------   Section 1: Universal $L^2$-torsion for $CW$-complexes-------------------}

\section{Universal $L^2$-torsion for $CW$-complexes and manifolds}
\label{sec:Universal_torsion_for_CW-complexes}

We will define the universal $L^2$-torsion $\rho^{(2)}_u(X;\caln(G)) \in \Wh^w(G)$ for an
$L^2$-acyclic finite free $G$-$CW$-complex $X$ by applying the notion of the universal $L^2$-torsion
of Section~\ref{sec:Universal_torsion_for_chain_complexes}  to the cellular chain
complex.  We will present the basic properties of this invariant in
Theorem~\ref{the:Main_properties_of_the_universal_L2-torsion}.

If $G$ is a group such that there exists a connected $L^2$-acyclic finite free
$G$-$CW$-complex, then we will see that every element in $\Wh^w(G)$ occurs as
$\rho^{(2)}_u(X;\caln(G))$, see Lemma~\ref{lem:realizability_G-CW}.


\subsection{The universal $L^2$-torsion for $G$-$CW$-chain complexes}
\label{subsec:The_universal_torsion_for_G-CW-complexes}

Notice that the cellular $G$-$CW$-structure on a finite free $G$-$CW$-complex defines only
an equivalence class of $\IZ G$-bases on $C_*(X)$, where we call two $\IZ G$-basis $B$ and
$B'$ equivalent if there exists a bijection $\sigma \colon B \to B'$ such that for every
$b \in B$ there exists $g\in G$ and $\epsilon\in \{\pm 1\}$ with $\epsilon \cdot g \cdot \sigma(b) = b$.  
The Hilbert $\caln(G)$-chain complex $C_*^{(2)}(X)$ is independent of the
choice of a $\IZ G$-basis within the equivalence class of cellular $\IZ G$-bases. So we can
define the $n$-th $L^2$-Betti number $b_n^{(2)}(X;\caln(G))$ to be
$b_n^{(2)}(C_*^{(2)}(X); \caln(G))$.

\begin{definition}[Universal $L^2$-torsion for $G$-$CW$-complexes] 
\label{def:universal_L2-torsion_for_G-CW-complexes}
Let $X$ be a finite free $G$-$CW$-complex which is $L^2$-acyclic, i.e., its $n$-th $L^2$-Betti number
$b_n^{(2)}(X;\caln(G))$ vanishes for all $n \ge 0$.
Then we define its \emph{universal $L^2$-torsion}
\[
\rho^{(2)}_u(X;\caln(G)) \in \Wh^w(G)
\]
to be the image of $\rho^{(2)}_u(C_*(X);\caln(G))$ under the projection $\widetilde{K}_1^w(\IZ G) \to \Wh^w(G)$
after any choice of $\IZ G$-basis for $C_*(X)$ which represents the equivalence class of cellular $\IZ G$-basis.

This definition extends to pairs $(X,Y)$ of finite free $G$-$CW$-complexes in the obvious way,
consider the cellular $\IZ G$-$CW$-complex $C_*(X,Y)$ and require that
$b_n^{(2)}(X,Y;\caln(G))$ vanishes for all $n \ge 0$.
\end{definition}

\begin{remark}[Universal $L^2$-torsion for manifolds]
  Every compact topological manifold has a preferred simple homotopy type,
  see~\cite[IV]{Kirby-Siebenmann(1969)}. We
  can therefore extend the definition of the universal $L^2$-torsion from CW-complexes to
  manifolds in the usual way.
\end{remark}

Since we consider the universal $L^2$-torsion as an element in $\Wh^w(G)$,
the choice of the $\IZ G$-basis representing the equivalence class of cellular $\IZ G$-basis
does not matter and $\rho^{(2)}_u(X;\caln(G))$ depends only on the finite free $G$-$CW$-structure on $X$.

\begin{example}[$G = \IZ$] \label{exa:G_is_Z} The quotient field $\IQ(z,z^{-1})$ of the
  integral domain $\IZ[z,z^{-1}] = \IZ[\IZ]$ consists of rational functions with
  coefficients in $\IQ$ in one variable $z$.  Given a matrix $A \in M_{n,n}(\IZ[\IZ])$,
  the operator $\Lambda\bigl(r_A \colon \IZ]\IZ]^n \to \IZ[\IZ]^n\bigr)$ is a weak
  isomorphism if and only if $\det_{\IZ[\IZ]}(A) \in \IZ[\IZ]$ is non-zero. This follows
  from~\cite[Lemma~1.34 on page~35]{Lueck(2002)}.  Hence we obtain a well-defined
  homomorphism
\[
{\det}_{\IZ[\IZ]} \colon K_1^w(\IZ[\IZ]) \to \IQ(z,z^{-1})^{\times},
\quad \bigl[r_A \colon \IZ[\IZ]^n \to \IZ[\IZ]^n\bigr] \mapsto {\det}_{\IZ[\IZ]}(A).
\]
Define 
\[
i \colon \IQ(z,z^{-1})^{\times} \to K_1^w(\IZ[\IZ])
\]
by sending an element $x \in \IQ(z,z^{-1})^{\times}$ to $[r_p \colon \IZ[\IZ] \to
\IZ[\IZ]] - [r_q \colon \IZ[\IZ] \to \IZ[\IZ]]$ for any two elements $p, q \in \IZ[\IZ]$
with $p \not= 0, q \not= 0$ satisfying $x = p \cdot q^{-1}$ in $\IQ(z,z^{-1})^{\times}$. One easily 
checks that $i$ is well-defined and ${\det}_{\IZ[\IZ]} \circ i = \id_{\IQ(z,z^{-1})^{\times}}$. It is not
hard to prove using the standard Euclidean algorithm on the polynomial ring $\IQ[z]$ 
that $i$ is surjective.  Hence ${\det}_{\IZ[\IZ]}$ and $i$ are inverses of each other.  We obtain an induced isomorphism
\begin{eqnarray}
\label{equ:whw}
\Wh^w(\IZ) 
& \xrightarrow{\cong} &
\IQ(z,z^{-1})^{\times} /\{\sigma \cdot  z^n \mid \sigma \in \{\pm 1\}, n \in \IZ\}.
\label{mir_faellt_kein_label_ein}
\end{eqnarray}
Consider $\IR$ with the standard $\IZ$-action given by translation. This is the universal
covering of $S^1$ with the standard action of $\pi_1(S^1) = \IZ$. The cellular
$\IZ[\IZ]$-chain complex of $\IR$ is $1$-dimensional and has as first differential
$r_{z-1} \colon \IZ[\IZ] \to \IZ[\IZ]$. Hence  the isomorphism~\eqref{mir_faellt_kein_label_ein} sends
$\rho^{(2)}_u(\IR;\caln(\IZ)) \in \Wh^w(\IZ)$ to the class of the element $(z-1)$.

More generally, let $X$ be any finite free $\IZ$-$CW$-complex. Then $X$ is $L^2$-acyclic
if and only if all of the $\IQ[\IZ]$-modules $H_n(X;\IQ)$ are torsion $\IQ[\IZ]$-modules,
see~\cite[Lemma~1.34 on page~35]{Lueck(2002)}.  It is now straightforward to see that
under the isomorphism~\ref{equ:whw} the universal $L^2$-torsion of $X$ corresponds to the
Milnor-Turaev torsion~\cite{Milnor(1966),Turaev(2001)} of $X$.
\end{example}

\begin{theorem}[Main properties of the universal $L^2$-torsion]\ 
\label{the:Main_properties_of_the_universal_L2-torsion}

\begin{enumerate}[font=\normalfont,leftmargin=0.9cm]
\item \emph{($G$-homotopy invariance)}
 \label{the:Main_properties_of_the_universal_L2-torsion:G-homotopy_invariance}
Let $f \colon X \to Y$ be a $G$-homotopy equivalence of finite free $G$-$CW$-complexes.
Suppose that $X$ or $Y$ is $L^2$-acyclic. Then both $X$ and $Y$ are $L^2$-acyclic and we get
\[
\rho^{(2)}_u(Y) - \rho^{(2)}_u(X)
 = \zeta(\tau(f)),
\]
where $\tau(f) \in \Wh(G)$ is the Whitehead torsion of $f$ and $\zeta \colon \Wh(G) \to \Wh^w(G)$
is the obvious homomorphism;

\item \emph{(Sum formula)}
\label{the:Main_properties_of_the_universal_L2-torsion:sum_formula}
Consider a $G$-pushout of  finite free $G$-$CW$-complexes
    \[
    \xymatrix@R0.5cm{ X_0 \ar[r] \ar[d] & X_1 \ar[d]
      \\
      X_2 \ar[r]& X }
    \]
    where the upper horizontal arrow is cellular, the left vertical arrow is an inclusion
    of $G$-$CW$-complexes and $X$ has the obvious $G$-$CW$-structure coming from the ones
    on $X_0$, $X_1$ and $X_2$.  Suppose that $X_0$, $X_1$ and $X_2$ are $L^2$-acyclic.
Then $X$ is $L^2$-acyclic and we get
    \[
   \hspace{1cm} \rho^{(2)}_u(X;\caln(G)) = \rho^{(2)}_u (X_1;\caln(G)) + \rho^{(2)}_u(X_2;\caln(G)) -
    \rho^{(2)}_u(X_0;\caln(G));
    \]

\item \emph{(Induction)}
   \label{the:Main_properties_of_the_universal_L2-torsion:induction}
   Let $j \colon H \to G$ be an inclusion of groups. 
   Denote by $j_* \colon \Wh^w(H) \to \Wh^w(G)$ the induced homomorphism.
    Let $X$ be an $L^2$-acyclic finite free $H$-$CW$-complex.
   Then $j_* X = G \times_H X$ is an $L^2$-acyclic finite free $G$-$CW$-complex and we get
    \[
   \rho^{(2)}_u(j_* X;\caln(G)) = j_* \bigl(\rho^{(2)}_u(X;\caln(H))\bigr);
    \]
   
\item \emph{(Restriction)}  
  \label{the:Main_properties_of_the_universal_L2-torsion:restriction}
   Let $j \colon H \to G$ be an inclusion of groups such that the index $[G:H]$ is finite. 
  Denote by $j^* \colon \Wh^w(G) \to \Wh^w(H)$ the homomorphism given by restriction.
    Let $X$ be a finite free $G$-$CW$-complex. Let $j^* X$ be the restriction
   to $H$ by $j$.
   Then $j^*X$ is a finite free $H$-$CW$-complex, $X$ is $L^2$-acyclic if and only if $j^*
   X$ is $L^2$-acyclic,  and, if this is the case,
   we get
    \[
   \rho^{(2)}_u(j^* X;\caln(H)) = j^* \bigl(\rho_u^{(2)}(X;\caln(G))\bigr);
    \]

\item \emph{(Product formula)}
    \label{the:Main_properties_of_the_universal_L2-torsion:product_formula}
    Let $G_0$ and $G_1$ be groups and denote by $j \colon G_0 \to G_0 \times G_1$ the obvious inclusion.
   Denote by $j_* \colon \Wh^w(G_0) \to \Wh^w(G_0 \times G_1)$ the induced homomorphism.
    Let $X_i$ be a finite free $G_i$-$CW$-complex for $i = 0,1$. Suppose that
    $X_0$ is $L^2$-acyclic. 
Then $X_0 \times X_1$ is $L^2$-acyclic and we get
    \[
    \rho^{(2)}_u(X_0 \times X_1; \caln(G_0 \times G_1)) 
   = \chi(X_1/G_1)  \cdot  j_* \bigl(\rho^{(2)}_u(X_0; \caln(G_0))\bigr);
    \]

\item \emph{(Fibrations)}
\label{the:Main_properties_of_the_universal_L2-torsion:fibrations}
  Let $F \xrightarrow{i} E \xrightarrow{p} B$ be a fibration. 
  Suppose that $F$ and $B$ are finite  $CW$-complexes.  
  Let $q \colon \overline{E} \to E$ be a $G$-covering. 
  Let $\overline{F} \to F$ be the $G$-covering obtained from $q$ by
  the pullback construction with $i$. Suppose that
  $\overline{F}$ is $L^2$-acyclic. Assume that $\Wh(G)$ vanishes.
Then $\overline{E}$ up to $G$-homotopy equivalence and $\overline{F}$ are $L^2$-acyclic
  finite free $G$-$CW$-complexes and we get in $\Wh^w(G)$
  \[
\rho^{(2)}_u(\overline{E};\caln(G)) =
\chi(B) \cdot \rho^{(2)}_u(\overline{F};\caln(G));
\]

\item  \emph{($S^1$-actions)}
\label{the:Main_properties_of_the_universal_L2-torsion:S1-actions}
Let $X$ be a connected finite $S^1$-$CW$-complex. Fix a base point $x \in X$.
Let $\mu \colon \pi_1(X,x) \to G$ be a group homomorphism.  Suppose that the composite
\[
\pi_1(S^1,1) \xrightarrow{\pi_1(\ev_x,1)} \pi_1(X,x) \xrightarrow{\mu} G
\]
is injective, where $\ev_x \colon S^1 \to X$ sends $z$ to $z \cdot x$.  Let $I_n$ be the
set of open $n$-dimensional $S^1$-cells of $X$. For each $e \in I_n$ choose a point $x(e)$ in
its interior and a path $w(e)$ from $x(e)$ to $x$. Denote by $S^1_{x(e)} \subseteq S^1$ the
isotropy group of $x(e)$ which must be finite because of the assumptions above.  
Let $\overline{\ev}_e \colon S^1/S^1_{x(e)} \to X$ be the
injective map which sends $z\cdot S^1_{x(e)}$ to $z \cdot x(e)$.  
Choose a homeomorphism $f_e \colon S^1 \xrightarrow{\cong} S^1/S^1_{x(e)}$ with 
$f_e(1) =1 \cdot S^1_{x(e)}$. Identify $\IZ = \pi_1(S^1,1)$.
Let $j(e) \colon \IZ \to G$ be the injective homomorphism given by the
composite
\begin{multline*}
\quad \quad \quad j(e) \colon  \IZ = \pi_1(S^1,1) \xrightarrow{\pi_1(f_e,1)} \pi_1(S^1/S^1_{x(e)},1\cdot S^1_{x(e)}) 
 \xrightarrow{\pi_1(\overline{\ev}_e,1\cdot S^1_{x(e)})}
\\
\pi_1(X, x(e)) \xrightarrow{c_{w(e)}} \pi_1(X,x) \xrightarrow{\mu} G,
\end{multline*}
where $c_{w(e)}$ is given by conjugation with the path $w(e)$. 
Denote by 
\[
j(e)_* \colon \Wh^w(\IZ) \to \Wh^w(G)
\]
the homomorphism induced by $j(e)$. Equip $\IR$ with the $\IZ$-action given by translation.
Then $\overline{X}$ is $L^2$-acyclic and we get
\[\rho^{(2)}_u(\overline{X};\caln(G)) = \sum_{n \ge 0} (-1)^n \cdot 
\sum_{e \in I_n} j(e)_* \bigl(\rho^{(2)}_u(\IR,\caln(\IZ))\bigr);
\]

\item \emph{(Poincar\'e duality)}
\label{the:Main_properties_of_the_universal_L2-torsion:Poincare_duality}
Let $M$ be an orientable $n$-dimensional manifold with free proper $G$-action and boundary $\partial M$.
Let $w \colon G \to \{\pm 1\}$ be the orientation homomorphism sending $g \in G$ to $1$, 
if multiplication with $g$ respects the orientation and to $-1$ otherwise. Equip $\Wh^w(G)$
with the involution $\ast$ coming from the $w$-twisted involution on $\IZ G$ sending
$\sum_{g \in G} r_g \cdot g$ to $\sum_{g \in G} r_g \cdot w(g) \cdot g^{-1}$. Then
\[
\rho^w_u(M,\partial M;\caln(G)) = (-1)^{n+1} \cdot \ast\bigl(\rho^w_u(M;\caln(G))\bigr).
\]
\end{enumerate}
\end{theorem}
\begin{proof}~\eqref{the:Main_properties_of_the_universal_L2-torsion:G-homotopy_invariance}
This follows from Lemma~\ref{lem:universal_L2-torsion_and_chain_homotopy_equivalences}.
\\[1mm]~\eqref{the:Main_properties_of_the_universal_L2-torsion:sum_formula}
This follows from Lemma~\ref{lem:Additivity_of_universal_L2-torsion}.
\\[1mm]~\eqref{the:Main_properties_of_the_universal_L2-torsion:induction}
This follows directly from the definitions and~\cite[Lemma~1.24~(4) on page~30]{Lueck(2002)} 
using the canonical isomorphism $\IZ G \otimes_{\IZ H} C_*(X) \cong C_*(j_*X)$.
\\[1mm]~\eqref{the:Main_properties_of_the_universal_L2-torsion:restriction}
Obviously $j^*X$ is a finite free $H$-$CW$-complex. We conclude
\[
b_n^{(2)}(j^*X;\caln(H)) = [G:H] \cdot b_n^{(2)}(X;\caln(G))
\]
from~\cite[Theorem~1.35~(9) on page~38]{Lueck(2002)}. Hence $X$ is $L^2$-acyclic if and
only if $j^* X$ is $L^2$-acyclic. Since the index $[G:H]$ is finite, there is an obvious
homomorphism $j^* \colon \Wh^w(G) \to \Wh^w(H)$ given by restriction with $j$. Now the
formula $\rho^{(2)}_u(j^* X;\caln(H)) = j^* (\rho^{(2)}_u(X;\caln(G)))$ follows from the
definitions since the restriction of the $\IZ G$-chain complex $C_*(X)$ to $\IZ H$ with $j$
agrees with $C_*(j^*X)$.
\\[1mm]~\eqref{the:Main_properties_of_the_universal_L2-torsion:product_formula}
This follows from
assertions~\eqref{the:Main_properties_of_the_universal_L2-torsion:G-homotopy_invariance},~%
\eqref{the:Main_properties_of_the_universal_L2-torsion:sum_formula}
and~\eqref{the:Main_properties_of_the_universal_L2-torsion:induction} by induction over
the equivariant cells in $Y$.
\\[1mm]~\eqref{the:Main_properties_of_the_universal_L2-torsion:fibrations} This
follows from
assertions~\eqref{the:Main_properties_of_the_universal_L2-torsion:G-homotopy_invariance}~%
and~\eqref{the:Main_properties_of_the_universal_L2-torsion:sum_formula} by induction over
the~equivariant cells in $B$.
\\[1mm]\eqref{the:Main_properties_of_the_universal_L2-torsion:S1-actions} 
This follows from assertions~\eqref{the:Main_properties_of_the_universal_L2-torsion:G-homotopy_invariance},~%
\eqref{the:Main_properties_of_the_universal_L2-torsion:sum_formula}
and~\eqref{the:Main_properties_of_the_universal_L2-torsion:induction} by induction over
the $S^1$-cells in $X$ using the fact that the finite free $\IZ$-$CW$-complex $\IR = \widetilde{S^1}$ 
is $L^2$-acyclic. 
\\[1mm]~\eqref{the:Main_properties_of_the_universal_L2-torsion:Poincare_duality}
There is a simple $\IZ G$-chain homotopy equivalence $C^{n-*}(M) \to C_*(M,\partial M)$, where
$C^{n-*}(M)$ is the dual $\IZ G$-chain complex with respect to the $w$-twisted involution.
Lemma~\ref{lem:universal_L2-torsion_and_chain_homotopy_equivalences} implies
\[
\rho^w_u(C^{n-*}(M)) = \rho^w_u(C_*(M,\partial M)).
\]
We conclude directly from the definitions.
\[
\rho^w_u(C^{n-*}(M)) = (-1)^{n+1} \cdot \ast\bigl(\rho^w_u(C_*(M)).
\]
This finishes the proof of Theorem~\ref{the:Main_properties_of_the_universal_L2-torsion}.
\end{proof}

\begin{remark}[Universal $L^2$-torsion in terms of the combinatorial Laplace operator]
\label{rem:Universal_L2-torsion_in_terms_of_the_combinatorial_Laplace_operator}
Let $M$ be an orientable $n$-dimensional manifold with free proper $G$-action and empty boundary
such that the $G$ action is orientation preserving. Suppose that the dimension of $M$ is odd.
Let $\Delta_n \colon C_*(M) \to C_*(M)$ be its combinatorial  Laplace operator.
Then we conclude from 
Lemma~\ref{lem:rho(2)_u(C_aat)_in_terms_of_the_combinatorial_Laplace_operator}
and Theorem~\ref{the:Main_properties_of_the_universal_L2-torsion}~%
\eqref{the:Main_properties_of_the_universal_L2-torsion:Poincare_duality}
that we get in $\Wh^w(G)$
\[
2 \cdot \rho^{(2)}_u(M;\caln(G)) = - \sum_{n \in \IZ} (-1)^n \cdot n \cdot [\Delta_n].
\]
The advantage of the formula above is that one can derive the right hand side directly from
the differentials without having to find  an explicit weak chain contraction. Notice that in the
case, where the dimension of $M$ is even, we do not get interesting information about the
universal $L^2$-torsion, namely, we just get
\[
0 = - \sum_{n \in \IZ} (-1)^n \cdot n \cdot  [\Delta_n].
\]
\end{remark}

\begin{example}[Torus $T^n$]
\label{exa:torus}
Let $T^n$ be the $n$-dimensional torus for $n \ge 2$.  Let $G$ be a torsion-free group. Let
$\mu \colon \pi_1(T^n) \to G$ be a non-trivial group homomorphism. Let $\overline{T^n} \to T^n$ be the
$G$-covering associated to $\mu$. Then $\overline{T^n}$ is $L^2$-acyclic and we get
\[
\rho^{(2)}_u(\overline{T^n};\caln(G)) = 0
\]
by the following argument. 

We can choose an integer $k$ with $k \ge 1$, an isomorphism 
$\nu \colon \pi_1(T^n) \xrightarrow{\cong} \IZ^k \times \IZ^{n-k}$ and an injection 
$i \colon \IZ^k \to G$ such that $\mu = i \circ \pr \circ \nu$ for 
$\pr \colon \IZ^k \times \IZ ^{n-k} \to \IZ^k$ the projection.  We can find a homeomorphism 
$f \colon T^n \xrightarrow{\cong} T^k \times T^{n-k}$ such that $\pi_1(f) = \nu$.  
Hence we obtain a $G$-homeomorphism
\[
i_*\bigl(\widetilde{T^k} \times T^{n-k}\bigr) \xrightarrow{\cong} \overline{T^n},
\]
where the $\IZ^k$-action on $\widetilde{T^k} \times T^{n-k}$ is given by the standard
$\IZ^k$-action on $\widetilde{T^k}$ and the trivial $\IZ^k$-action on $T^{n-k}$. We
conclude from Theorem~\ref{the:Main_properties_of_the_universal_L2-torsion}~%
\eqref{the:Main_properties_of_the_universal_L2-torsion:induction}
and~\eqref{the:Main_properties_of_the_universal_L2-torsion:product_formula} that
$\overline{T^n}$ is $L^2$-acyclic and
\[
\rho^{(2)}_u(\overline{T^n};\caln(G))
\,\,= \,\,
\rho^{(2)}_u(\widetilde{T^k} \times T^{n-k};\caln(\IZ^k))
\,\,=\,\,
\chi(T^{n-k}) \cdot \rho^{(2)}_u(\widetilde{T^k};\caln(\IZ^k)).\]
If $k \not= n$, then the claim follows from $\chi(T^{n-k}) = 0$. Suppose that $k = n$.
Then we have to show $\rho^{(2)}_u(\widetilde{T^k};\caln(\IZ^k)) = 0$ for $k \ge 2$. 
This follows from Theorem~\ref{the:Main_properties_of_the_universal_L2-torsion}~%
~\eqref{the:Main_properties_of_the_universal_L2-torsion:product_formula} 
applied to $\widetilde{S^1} \times \widetilde{T^{k-1}}$ using $\chi(T^{k-1}) = 0$.
\end{example}

\begin{lemma}[Realizability of the universal $L^2$-torsion for $G$-$CW$-complexes] \label{lem:realizability_G-CW}
Let $G$ be a group such that there exists a connected $L^2$-acyclic finite free $G$-$CW$-complex
$X$. Consider any element $\omega \in \Wh^w(G)$.  Then there exists a connected  $L^2$-acyclic
finite free $G$-$CW$-complex $Y$ obtained from $X$ by attaching trivially
equivariant $2$-cells and attaching equivariant $3$-cells with 
$\rho^{(2)}_u(Y;\caln(G)) = \omega$.
\end{lemma}

\begin{proof}
Choose a matrix $A \in M_{n,n}(\IZ G)$ such that $\Lambda(r_A)$ is a weak isomorphism
and $[r_A] = \omega - \rho^{(2)}_u(X;\caln(G))$ holds in $\Wh^w(G)$. 
Choose base points
$x \in X$ and $s \in S^2$.  Let $k \colon G \to X$ be the $G$-map sending $g$ to $g \cdot x$ 
and let $l \colon G \to G \times S^2$ be the $G$-map sending $g$ to $(g,s)$.  Let $X'$
 be the finite free $G$-$CW$-complex $X'$ given by the $G$-pushout
\[
\xymatrix@R0.6cm{\coprod_{j = 1}^n  G \ar[r]^-{\coprod_{i = 1}^n k} \ar[d]_{\coprod_{j = 1}^n l}
& X \ar[d]
\\
\coprod_{j = 1} G \times S^2 \ar[r] 
&
X'.
}
\] 
For $g \in G$ choose a path $v_g$ in $X$ from $x$ to $gx$.  Let $t_g \colon \pi_2(X',gx)
\to \pi_2(X',x)$ be the standard isomorphism of abelian groups given by $v_g$.  Fix elements
$i,j \in \{1,2, \ldots, n\}$.  Let $a[i,j] = \sum_{g \in G} a[i,j]_g \cdot g \in \IZ G$ be the
entry of $A$ at $(i,j)$.  Choose for $g \in G$ a pointed map $q[i,j]_g \colon (S^2,s) \to (X',x)$ 
such that its  class $[q[i,j]_g]$ in $\pi_2(X',x)$ 
is $a[i,j]_g$-times the image under $t_g \colon \pi_2(X',gx) \to \pi_2(X',x)$ of the 
element in $\pi_2(X',gx)$ given by the composite $S^2 \to G \times S^2, y \mapsto (g,y)$ with 
the inclusion of the $j$-th summand of $\coprod_{j = 1} G \times S^n$ into $X'$. 
Let $q_i\colon (S^2,s) \to (X',x)$ be a pointed map representing in $\pi_2(X',x)$ 
the element $\sum_{j = 1}^n \sum_{g \in G} [q[i,j]_g]$. Let 
$\widehat{q_i} \colon G \times S^2 \to X$ be the $G$-map  sending $(g,z)$ to $g \cdot q_i(z)$. 
Define a finite free $G$-$CW$-complex  $Y$ by the $G$-pushout
\[
\xymatrix@!C=6em@R=0.5cm{\coprod_{i = 1}^n G \times S^2 \ar[r]^-{\coprod_{i = 1}^n \widehat{q_i}} \ar[d]
& X' \ar[d]
\\
\coprod_{i = 1}^n G \times D^3 \ar[r]
&
Y.
}
\] 
By construction there is a based  exact sequence of finite based free $\IZ G$-chain complexes 
$0 \to C_*(X) \to C_*(Y) \to D_* \to 0$, where $D_*$ is concentrated 
in dimensions $2$ and $3$ and has as third differential $r_A \colon \IZ G^n \to \IZ G^n$.
Since $\Lambda(C_*(X))$ and $\Lambda(D_*)$ 
are $L^2$-acyclic, Lemma~\ref{lem:Additivity_of_universal_L2-torsion} implies that also 
$\Lambda(C_*(Y))$ is $L^2$-acyclic and we get in $\Wh^w(G)$
\[
\begin{array}{rcl}
\rho^{(2)}_u(Y)
& = &
\rho^{(2)}_u(C_*(Y;\caln(G))\\
& = & 
\rho^{(2)}_u(C_*(X;\caln(G)) + \rho^{(2)}_u(D_*;\caln(G))
\\
& = & 
\rho^{(2)}_u(C_*(X;\caln(G)) + [r_A \colon \IZ G^n \to \IZ G^n]
\\
& = & 
\rho^{(2)}_u(C_*(X;\caln(G))  + \omega - \rho^{(2)}_u(C_*(X;\caln(G)) 
\,\, = \,\,
\omega.\qedhere
\end{array}
\]
\end{proof}

\begin{lemma}[Realizability of the universal $L^2$-torsion for manifolds without boundary
  and cocompact free proper $G$-action]
  \label{lem:realizability_manifolds}
  Let $G$ be a group such that there exists a connected $L^2$-acyclic finite free
  $G$-$CW$-complex $X$. Consider any element $\omega \in \Wh^w(G)$ and
  any integer $d \ge 2 \cdot \max\{\dim(X),3\} + 1$.  Then there exists a connected smooth
  $d$-dimensional manifold $M$ without boundary and cocompact free proper smooth
  $G$-action such that $M$ is $L^2$-acyclic and we have
  \[
  \rho^{(2)}_u(M;\caln(G)) = (-1)^{d+1} \cdot \ast(\omega) + \omega.
  \]
Furthermore, if $X$ is simply-connected, then $M$ can also be chosen to be simply-connected.
\end{lemma}

\begin{proof}
  By Lemma~\ref{lem:realizability_G-CW} we can find a connected $L^2$-acyclic finite free
  $G$-$CW$-complex $Y$ of dimension $\max\{3,\dim(X)\}$ with $\rho^{2}(Y;\caln(G)) =
  \omega$. We can embed $Y/G$ into $\IR^d$ and choose a regular neighborhood $N$.  This
  is a compact manifold $N$ with boundary $\partial N$ such that the inclusion $i \colon Y  \to N$ 
  is a simple homotopy equivalence. Let $\overline{N} \to N$ be the $G$-covering obtained from the
  $G$-covering $Y \to Y/G$ by the pullback construction applied to any homotopy inverse of
  $i \colon Y \to N$.  Let $\overline{\partial N} \to \partial N$ be the restriction of
  $\overline{N} \to N$ to $\partial N$. Since $Y$ is $L^2$-acyclic, $\overline{N}$ is
  $L^2$-acyclic. We conclude from Poincar\'e duality and the long weak exact
  $L^2$-homology sequence, see~\cite[Theorem~1.21 on page~27 and Theorem~1.35 on
  page~37]{Lueck(2002)} that $\overline{\partial N}$ and $(\overline{N},\overline{\partial N})$ 
  are $L^2$-acyclic. Let $M$ be $\overline{N} \cup_{\overline{\partial N}} \overline{N}$.
  This is a smooth manifold without boundary with proper free smooth
  $G$-action. One easily checks using Lemma~\ref{lem:Additivity_of_universal_L2-torsion} 
  and Theorem~\ref{the:Main_properties_of_the_universal_L2-torsion}
  that $N$ is $L^2$-acyclic and we get in $\Wh^w(G)$
  \begin{eqnarray*}
    \rho^{(2)}_u(M;\caln(G)) 
    & = &
    \rho_u^{(2)}(\overline{N};\caln(G)) -   \rho_u^{(2)}(\overline{\partial N};\caln(G)) + \rho_u^{(2)}(\overline{N};\caln(G)) 
    \\
    & = & 
    \rho_u^{(2)}(\overline{N};\caln(G)) + \rho_u^{(2)}(\overline{N}, \overline{\partial N};\caln(G)) 
    \\
    & = & 
    \rho_u^{(2)}(\overline{N};\caln(G)) + (-1)^{d+1} \cdot \ast\bigl(\rho_u^{(2)}(\overline{N};\caln(G))\bigr)
    \\
    & = &\omega + (-1)^{d+1} \cdot \ast(\omega).
  \end{eqnarray*}
  Notice that for an element $\omega \in \Wh^w(G)$ the equality $\omega = (-1)^{d+1} \cdot \ast(\omega)$ 
  is a necessary condition for $\omega$  to be realized as
  $\omega = \rho^{(2)}_u(M;\caln(G))$ for a smooth orientable manifold $M$ without boundary and
  proper free orientation preserving $G$-action such that $M$ is $L^2$-acyclic,
 see Theorem~\ref{the:Main_properties_of_the_universal_L2-torsion}~%
\eqref{the:Main_properties_of_the_universal_L2-torsion:Poincare_duality}.
 The construction above shows that for an element $\omega \in \Wh^w(G)$
  satisfying $\omega = (-1)^{d+1} \cdot \ast(\omega)$ we can find a  smooth orientable manifold $M$
  without boundary and proper free orientation preserving $G$-action such that $M$ is $L^2$-acyclic and
  $\rho^{(2)}_u(M;\caln(G)) = 2 \cdot \omega$.
  
  Finally, if $X$ is simply-connected, then $\overline{N}$ is simply-connected and $M$ is also simply-connected.
\end{proof}


\subsection{The universal $L^2$-torsion for universal coverings}
\label{subsec:The_universal_torsion_for_universal_coverings}

The most natural and interesting case is the one of a universal covering. For the reader's
convenience we record the basic properties of the universal $L^2$-torsion in this setting.

\begin{definition}[Universal $L^2$-torsion for universal coverings]
  \label{def:universal_L2-torsion_for_universal_coverings}
  Let $X$ be a finite connected $CW$-complex. We call it \emph{$L^2$-acyclic} if its universal covering
  $\widetilde{X}$ is $L^2$-acyclic, i.e., the $n$-th $L^2$-Betti number
  $b_n^{(2)}(\widetilde{X}) := b_n^{(2)}(\widetilde{X};\caln(\pi_1(X)))$ vanishes for all
  $n \ge 0$.  Then we define its \emph{universal $L^2$-torsion}
  \[
  \rho^{(2)}_u(\widetilde{X}) \in \Wh^w(\pi_1(X))
  \]
  by $\rho^{(2)}_u(\widetilde{X};\caln(\pi_1(X)))$ as introduced in
  Definition~\ref{def:universal_L2-torsion_for_G-CW-complexes}.

  If $X$ is a finite $CW$-complex, we call it \emph{$L^2$-acyclic} if each path component 
  $C \in   \pi_0(C)$ is $L^2$-acyclic in the sense above and we define
\begin{eqnarray*}
\Wh^w(\Pi(X)) 
& := & 
\bigoplus_{C \in \pi_0(X)} \Wh^w(\pi_1(C));
\\
\rho^{(2)}_u(\widetilde{X})  & := & \bigl(\rho^{(2)}_u(\widetilde{C})\bigr)_{C \in \pi_0(X)} \quad \in \Wh^w(\Pi(X)).
\end{eqnarray*}

This definition extends to $CW$-pairs $(X,A)$
by
\[\rho^{(2)}_u(\widetilde{X},\widetilde{A})  
:=  \bigl(\rho^{(2)}_u(\widetilde{C}),\widetilde{A \cap C}\bigr)_{C \in \pi_0(X)} \quad \in \Wh^w(\Pi(X)),
\] 
where we denote  for a path component $C$ of $X$ by $\widetilde{A \cap C} \to A \cap C$
 the restriction of the universal covering $\widetilde{C} \to C$ to $A \cap C$.
\end{definition}

Given a map $f \colon X \to Y$ of finite $C$-complexes such that 
$\pi_1(f,x) \colon \pi_1(X,x) \to \pi_1(Y,f(x))$ is injective for all $x \in X$, we get a homomorphism 
$f_* \colon \Wh^w(\Pi(X)) \to \Wh^w(\Pi(Y)$ by the collection of maps 
$(f|_C)_* \colon \Wh^w(\pi_1(C)) \to \Wh^w(\pi_1(D))$ for $C \in \pi_0(X)$ 
and $D \in \pi_1(Y)$ with $f(C) \subseteq D$.

It is evident that Theorem~\ref{the:Main_properties_of_the_universal_L2-torsion} gives
rise to statements about the universal $L^2$-torsions for universal coverings.  Most
statements of Theorem~\ref{the:Main_properties_of_the_universal_L2-torsion} specialize in
an obvious way. Therefore in the next theorem we spell out only three properties.

\begin{theorem}[Main properties of the universal $L^2$-torsion for universal coverings]\ 
\label{the:Main_properties_of_the_universal_L2-torsion_for_universal_coverings}
\begin{enumerate}
\item[$(2)$] \emph{(Sum formula)}
Consider a pushout of  finite $CW$-complexes
    \[
    \xymatrix@R0.6cm{ X_0 \ar[r] \ar[rd]^{j_0} \ar[d] & X_1  \ar[d]^{j_1}
      \\
      X_2 \ar[r]^{j_2} & X }
    \]
    where the upper horizontal arrow is cellular, the left vertical arrow is an inclusion
    of $CW$-complexes and $X$ has the obvious $CW$-structure coming from the ones
    on $X_0$, $X_1$ and $X_2$.  Suppose that $X_0$, $X_1$ and $X_2$ are $L^2$-acyclic
    and  that for $i = 0,1,2$ and any point $x_i \in X$ the homomorphism 
     $\pi_1(j_i,x_i) \colon \pi_1(X_i,x_i) \to \pi_1(X,j_i(x_i))$ is injective.
 Then $X$ is $L^2$-acyclic and we get
    \[
  \hspace{1cm}  \rho^{(2)}_u(\widetilde{X}) = (j_1)_*\bigl(\rho^{(2)}_u (\widetilde{X_1})\bigr) + 
    (j_2)_*\bigl(\rho^{(2)}_u(\widetilde{X_2})\bigr) -     (j_0)_*\bigl(\rho^{(2)}_u(\widetilde{X_0})\bigr).
    \]
\item[$(5)$] \emph{(Fibrations)}
  Let $F \xrightarrow{i} E \xrightarrow{p} B$ be a fibration. Suppose that $F$and $B$ are connected finite $CW$-complexes.  
  Assume that $\pi_1(i) \colon \pi_1(F) \to \pi_1(E)$ is injective and that  $F$ is $L^2$-acyclic. Suppose that $\Wh(\pi_1(E))$ vanishes.
  Then $E$ is up to homotopy an $L^2$-acyclic connected finite $CW$-complex  and we get
  \[
\rho^{(2)}_u(\widetilde{E}) =
\chi(B) \cdot i_*\bigl(\rho_u^{(2)}(\widetilde{F})\bigr).
\]
\item[$(6)$]  \emph{($S^1$-actions)}
Let $X$ be a connected finite $S^1$-$CW$-complex. If we use the notation and make the assumptions
of Theorem~\ref{the:Main_properties_of_the_universal_L2-torsion}~\eqref{the:Main_properties_of_the_universal_L2-torsion:S1-actions}
in the special case $G = \pi_1(X)$ and $\mu = \id_{\\pi_1(X)}$, then $X$ is $L^2$-acyclic and we get
\[
\rho^{(2)}_u(\widetilde{X}) 
= \sum_{n \ge 0} (-1)^n \cdot  \sum_{e \in I_n} j(e)_* \bigl(\rho^{(2)}_u(\widetilde{S^1})\bigr).
\]
\end{enumerate}
\end{theorem}

\begin{remark}[Realizability of the universal $L^2$-torsion for universal coverings.]
\label{rem:realizability_universal_coverings}
Let $\pi$ be group such that there exists a connected finite $CW$-complex $X$ with 
$\pi =  \pi_1(X)$ which is $L^2$-acyclic. Consider any
element $\omega \in \Wh^w(\pi)$. As a special case of
Lemma~\ref{lem:realizability_G-CW} we get
that there is a connected finite $CW$-complex $Y$ with $\pi = \pi_1(Y)$
such that $Y$ is $L^2$-acyclic and $\rho^{(2)}_u(\widetilde{Y}) = \omega$.

Also Lemma~\ref{lem:realizability_manifolds} has an obvious analogue for $\rho^{(2)}_u(\widetilde{M})$
for a closed manifold  $M$.
\end{remark}

\begin{example}[$\rho^{(2)}_u(\widetilde{T^n})$]
\label{exa:rho(2)_u(widetilde(Tn)}
If $n \ge 2$, we conclude $\rho^{(2)}_u(\widetilde{T^n}) = 0$ from Example~\ref{exa:torus}.
We have computed $\rho^{(2)}_u(\widetilde{S^1})$ in Example~\ref{exa:G_is_Z}.
\end{example}

\begin{theorem}[Jaco-Shalen-Johannson decomposition]
\label{lem:JSJ_decom} 
   Let $M$ be an admissible  $3$-manifold and let
  $M_1$, $M_2$, \ldots, $M_r$ be its pieces in the Jaco-Shalen-Johannson decomposition. 
  Let $j_i \colon \pi_1(M_i) \to \pi_1(M)$ be the injection induced by the inclusion $M_i \to M$.
  Then each $M_i$ and $M$ are $L^2$-acyclic  and we have
\[
\rho^{(2)}_u(\widetilde{M}) \,=\, \sum_{i = 1}^r (j_i)_*\bigl(\rho^{(2)}_u(\widetilde{M_i})\bigr).
\]
\end{theorem}
\begin{proof}
Each piece $M_i$ is $L^2$-acyclic by~\cite[Theorem~0,1]{Lott-Lueck(1995)}.
Now the claim follows from Example~\ref{exa:torus} and
Theorem~\ref{the:Main_properties_of_the_universal_L2-torsion_for_universal_coverings}~(2).
\end{proof}

\begin{remark}[Graph manifolds]
  \label{rem:graph_manifolds}
We recall that a \emph{Seifert manifold} is loosely speaking a singular $S^1$-bundle over a surface. A \emph{graph manifold} is a 3-manifold that is obtained by gluing Seifert manifolds along boundary tori.
We refer to \cite{Aschenbrenner-Friedl-Wilton(2015)} for precise definitions.
  There is an obvious analogue of
  Theorem~\ref{the:Main_properties_of_the_universal_L2-torsion_for_universal_coverings}~(6)
  for Seifert manifolds $M$ with infinite fundamental groups where the role of $S^1$ is
  played by the regular fiber whose inclusion to $M$ always induces an injection on the
  fundamental groups, and the cells corresponds to tubular neighborhoods of fibers. So we
  get again a formula for $\rho^{(2)}_u(\widetilde{M})$ which essentially reduces the
  computation to the one of $\rho^{(2)}_u(\widetilde{S^1})$. In view of
  Theorem~\ref{lem:JSJ_decom} this extends to graph manifolds.  The hyperbolic pieces in
  the Jaco-Shalen-Johannson decomposition are much harder to deal with.
\end{remark}


\subsection{Mapping tori}
\label{subsec:Mapping_tori}

  Let $f \colon X   \to X$ be a self-map of a connected finite $CW$-complex.  
  Denote by  $T_f$ the mapping torus. For $p \colon T_f \to S^1$ the obvious projection, 
  consider any   factorization 
  $\pi_1(p) \colon \pi_1(T_f) \xrightarrow{\mu} G \xrightarrow{\phi}  \pi_1(S^1) = \IZ$.  
  Let $\overline{T_f} \to T_f$ be the $G$-covering associated to  $\mu  \colon \pi_1(T_f) \to G$.
  Then $\overline{T_f}$ is $L^2$-acyclic by~\cite[Theorem~2.1]{Lueck(1994b)} and we can consider
  \begin{eqnarray}
  & \rho^{(2)}_u(\overline{T_f};\caln(G)) \in K_1^w(\IZ G), &
    \label{rho(2)_u(overline(mapping-torus),caln(G))}
  \end{eqnarray}
  and especially
  \begin{eqnarray}
  & \rho^{(2)}_u(\widetilde{T_f}) \in K_1^w(\IZ [\pi_1(T_f)]. &
   \label{rho(2)_u(widetilde(mapping-torus))}
 \end{eqnarray}
 
  In particular the latter is an interesting invariant of $f$.
  The following makes it possible to reduce the complexity of calculating the invariant for mapping tori.

  \begin{lemma} \label{lem:mapping_tori}

  \begin{enumerate}[font=\normalfont]

  \item  \label{lem:mapping_tori:additivity}
  Consider a pushout of  finite connected $CW$-complexes
    \[
    \xymatrix@R0.6cm{ X_0 \ar[r]^{l_1} \ar[rd]^{j_0} \ar[d]^{l_2} & X_1 \ar[d]^{j_1}
      \\
      X_2 \ar[r]^{j_2} & X }
    \]
    where the upper horizontal arrow is cellular, the left vertical arrow is an inclusion
    of $CW$-complexes and $X$ has the obvious $CW$-structure coming from the ones
    on $X_0$, $X_1$ and $X_2$.  Suppose that the homomorphism 
     $\pi_1(j_i) \colon \pi_1(X_i) \to \pi_1(X)$ is injective for $i = 0,1,2$. Consider
     self-homotopy equivalences  $f_i \colon X_i \to X_i$ satisfying $f_i \circ l_i = l_i \circ f_0$ for $i = 1,2$.
     Let $f \colon X \to X$ be the self homotopy equivalence determined by the pushout property.
 
     Then we obtain a pushout of connected finite $CW$-complexes
     \[
    \xymatrix@R0.6cm{T_{f_0} \ar[r] \ar[rd]^{k_0} \ar[d] & T_{f_1} \ar[d]^{k_1}
      \\
      T_{f_2} \ar[r]^{k_2} & T_f}
    \]
     such that $\pi_1(k_i)$ is injective for $i = 0,1,2$, and we get
     \[
      \rho^{(2)}_u(\widetilde{T_f}) 
       = 
      (k_1)_* \bigl(\rho^{(2)}_u(\widetilde{T_{f_1}})\bigr) 
      + (k_2)_* \bigl(\rho^{(2)}_u(\widetilde{T_{f_2}})\bigr) - (k_0)_* \bigl(\rho^{(2)}_u(\widetilde{T_{f_0}})\bigr);
      \]
   
      \item  \label{lem:mapping_tori:l2-acyclic_fiber} 

   If $f \colon X \to X$ is a self-homotopy equivalence of a connected finite $CW$-complex $X$ such that
      $X$ is $L^2$-acyclic, then
      \[
      \rho^{(2)}_u(\widetilde{T_f}) 
       =  
      0.
      \]
    \end{enumerate}
  \end{lemma}
  \begin{proof}~\eqref{lem:mapping_tori:additivity}
  This follows from
     Theorem~\ref{the:Main_properties_of_the_universal_L2-torsion_for_universal_coverings}~(2).
    \\[1mm]~\eqref{lem:mapping_tori:l2-acyclic_fiber} 
   This follows from Theorem~\ref{the:Main_properties_of_the_universal_L2-torsion_for_universal_coverings}~(2) applied to the
      pushout
      \[
      \xymatrix@R0.5cm{X \times \{0,1\} \ar[r]^-{\id_X \amalg f} \ar[d]
      &
      X \ar[d]
      \\
       X \times [0,1] \ar[r]
       &
       T_f.}
       \]
     \end{proof}


\subsection{$L^2$-torsion and the $L^2$-Alexander torsion}
\label{subsec:L2-torsion}
 \label{subsec:L2-torsion_twisted_with_finite-dimensional_representations}

Throughout this section let $M$ be an admissible $3$-manifold. We write $\pi=\pi_1(M)$.
Taking the Fuglede-Kadison determinant yields a homomorphism
\[
{\det}_{\caln(\pi)} \colon \Wh^w(\pi) \to \IR.
\]
The image of $\rho^{(2)}_u(\widetilde{M})$ under this homomorphism is easily seen to be the $L^2$-torsion
$\rho^{(2)}(\widetilde{M})$ which can be computed as $-1/6\pi$-times the sum of the
volumes of the hyperbolic pieces in the Jaco-Shalen-Johannson decomposition of $M$, 
see~\cite[Theorem~0.7]{Lueck-Schick(1999)}.

We can generalize this discussion.
  Let $\pr \colon \pi \to H_1(\pi)_f := H_1(\pi)/\tors(H_1(\pi))$ be the  projection.
  Denote by $\Rep_{\IC}(H_1(\pi)_f)$ the representation ring of finite-dimensional complex
  $H_1(\pi)_f$-representations.   There is a pairing
  \[
  \Wh^w(\pi) \otimes \Rep_{\IC}(H_1(\pi)_f) \to \IR
  \]
  given by the Fuglede-Kadison determinant twisted with $\pr^* V$ for some
  finite-dimensional $H_1(\pi)_f$-representation $V$.  It sends
  $\bigl(\rho^{(2)}_u(\widetilde{M}),[V]\bigr)$ to the $\pr^* V$-twisted $L^2$-torsion
    $\rho^{(2)}(\widetilde{M};\pr^*V)$, see~\cite{Lueck(2015twisting)}.

  Given $\phi \in H^1(M;\IZ) = \hom_{\IZ}(H_1(\pi)_f,\IZ)$, we get for every 
   $t    \in (0,\infty)$ an element in $\Rep_{\IC}(H_1(\pi)_f)$ by the $1$-dimensional
    representation $\IC_{t,\phi}$, which is given by the $H_1(\pi)_f$-action on $\IC$ determined by
   $g \cdot \lambda :=  t^{\phi(g)} \cdot \lambda$ for $g \in H_1(\pi)_f$ and $\lambda \in \IC$. 
   Thus we  obtain an $L^2$-torsion function 
   \begin{equation*}
   \rho^{(2)}(\widetilde{M}) \colon (0,\infty) \to \IR, 
   \quad t \mapsto \rho^{(2)}(\widetilde{M};\pr^*\IC_{t,\phi}).
   \end{equation*}
A standard argument shows that this function is a well-defined invariant of the pair $(M,\phi)$ up to the addition of a function of the form $t\mapsto k\ln(t)$ for some $k\in \IZ$.  The $L^2$-torsion function $ \rho^{(2)}(\widetilde{M})$
    which is determined by $\rho^{(2)}_u(\widetilde{M}) \in \Wh^w(\pi)
    $ and whose value at
    $t = 1$ is the $L^2$-torsion $\rho^{(2)}(\widetilde{M})$. This $L^2$-torsion function is studied
    for instance in~\cite{Dubois-Friedl-Lueck(2016Alexander), Friedl-Lueck(2015l2+Thurston),Herrmann(2016), 
    Li-Zhang(2006volume),Li-Zhang(2006Alexander),Liu(2015)}.
    Moreover, $\limsup_{t \to \infty} \frac{\rho^{(2)}(\widetilde{M})(t)}{\ln(t)}$ and 
    $\liminf_{t \to 0} \frac{\rho^{(2)}(\widetilde{M})(t)}{\ln(t)}$
    exist as real numbers and their difference is called the degree of the $L^2$-torsion function.
    The negative of the degree turns out to be the Thurston seminorm $x_M(\phi)$ of $\phi$, 
    see~\cite{Friedl-Lueck(2015l2+Thurston), Liu(2015)}.


\typeout{---------   Section 3: The Grothendieck group of integral  polytopes and $K_1^w(\IZ G)$  --------}

\section{The Grothendieck group of integral polytopes and $K_1^w(\IZ G)$}
\label{subsec:The_Grothendieck_group_of_integral_polytopes_and_K_1w(ZG)}

In this section we want to detect elements in $\Wh^w(G)$, in particular
$\rho^{(2)}_u(X;\caln(G))$, in terms of integral polytopes in $\IR \otimes_{\IZ} H_1(G)_f$
and establish for admissible $3$-manifolds different from $S^1 \times D^2$ a relation to the Thurston seminorm.


\subsection{The Grothendieck group of integral polytopes}
\label{subsec:The_Grothendieck_group_of_integral_polytope}

Next we recall the definition of and give some information about the polytope group from
Friedl-L\"uck~\cite[Section~6.2] {Friedl-Lueck(2016l2+poly)}.

A \emph{polytope} in a finite-dimensional real vector space $V$ is a subset which is the
convex hull of a finite subset of $V$.  An element $p$ in a polytope is called
\emph{extreme} if the implication $p= \frac{q_1}{2} + \frac{q_2}{2} \Longrightarrow q_1 = q_2 = p$ 
holds for all elements $q_1$ and $q_2$ in the polytope.  Denote by $\Ext(P)$ the
set of extreme points of $P$.  If $P$ is the convex hull of the finite set $S$, then
$\Ext(P) \subseteq S$ and $P$ is the convex hull of $\Ext(P)$.  The \emph{Minkowski sum} of two
polytopes $P_1$ and $P_2$ is defined to be the polytope
\[
P_1 + P_2 := \{p_1 +p_2 \mid p_1 \in P_1, p \in P_2\}.
\] 
It is the convex hull of the set $\{p_1 + p_2 \mid p_1 \in \Ext(P_1), p_2 \in \Ext(P_2)\}$.

Let $H$ be a finitely generated free abelian group. We obtain a finite-dimensional
real vector space $\IR \otimes_{\IZ} H$.  An integral polytope in $\IR \otimes_{\IZ} H$ is
a polytope such that $\Ext(P)$ is contained in $H$, where we consider $H$ as a lattice in
$\IR \otimes_{\IZ} H$ by the standard embedding 
$H \to \IR \otimes_{\IZ} H, \; h \mapsto 1 \otimes h$.  
The Minkowski sum of two integral polytopes is again an integral
polytope. Hence the integral polytopes form an abelian monoid under the Minkowski sum with
the integral polytope $\{0\}$ as neutral element. 

\begin{definition}[Grothendieck group of integral polytopes]
\label{def:The_Grothendieck_group_of_integral_polytope}
Let $\calp_{\IZ}(H)$ be the abelian group given by the Grothendieck construction applied
to the abelian monoid of integral polytopes in $\IR \otimes_{\IZ} H$ under the Minkowski sum.
\end{definition}

Notice that for polytopes $P_0$, $P_1$ and $Q$ in a finite-dimensional real vector space
we have the implication $P_0 + Q = P_1 + Q \Longrightarrow P_0 = P_1$, see~\cite[Lemma~2]{Radstroem(1952)}.
Hence elements in $\calp_{\IZ}(H)$ are given by formal
differences $[P] - [Q]$ for integral polytopes $P$ and $Q$ in $\IR \otimes_{\IZ} H$ and we
have $[P_0] - [Q_0] = [P_1] - [Q_1] \Longleftrightarrow P_0 + Q_1 = P_1 + Q_0$.

There is an obvious homomorphism of abelian groups $i \colon H \to \calp_{\IZ}(H)$
which sends $h \in H$ to the class of the polytope $\{h\}$. Denote its cokernel by
\begin{eqnarray}
\calp_{\IZ}^{\Wh}(H) 
& = & 
\coker\bigl(i \colon H\to \calp_{\IZ}(H)\bigr).
\label{calp_IZ,Wh}
\end{eqnarray}
Put differently, in $\calp_{\IZ}^{\Wh}(H)$ two polytopes are identified if they are obtained by
translation with some element in the lattice $H$ from one another.

\begin{example}\label{ex:polytopes-in-r}
An  integral polytope in $\IR \otimes_{\IZ} \IZ$ is given by an interval $[m,n]$ for
integers $m,n$ with $m \le n$. The Minkowski sum becomes
$[m_1,n_1] + [m_2,n_2] = [m_1 + m_2, n_1 +n_2]$. One easily checks that one obtains isomorphisms of
abelian groups
\begin{eqnarray}
\quad \quad \quad \calp_{\IZ}(\IZ) & \xrightarrow{\cong} & \IZ^2
\quad 
[[m,n]]  \mapsto (n-m,m); 
\label{calp_Z(Z)_is_Z2}
\\
\quad \quad \quad \calp_{\IZ}^{\Wh}(\IZ) & \xrightarrow{\cong} & \IZ, \quad [[m,n]] \mapsto n-m. 
\label{calp_Z,Wh(Z)_is_Z}
\end{eqnarray}
\end{example}

Given a  homomorphism of finitely generated abelian groups $f \colon H \to H'$,
we can assign to an integral polytope $P \subseteq \IR \otimes_{\IZ} H$ an integral polytope
in $\IR \otimes_{\IZ} H'$ by the image of $P$ under 
$\id_{\IR} \otimes_{\IZ} f \colon  \IR \otimes_{\IZ} H \to  \IR \otimes_{\IZ} H'$
and thus we obtain  homomorphisms of abelian groups
\begin{eqnarray}
\calp_{\IZ}(f) \colon \calp_{\IZ}(H) & \to & \calp_{\IZ}(H'),
\quad [P] \mapsto [\id_{\IR} \otimes_{\IZ} f(P)];
\label{calp_Z(f)}
\\
\calp_{\IZ}^{\Wh}(f) \colon \calp_{\IZ}^{\Wh}(H) & \to & \calp_{\IZ}^{\Wh}(H'). 
\label{calp_Z,Wh(f)}
\end{eqnarray}

\begin{lemma} \label{lem:structure_of_polytope_group}
Let $H$ be a finitely generated free abelian group. Then:

\begin{enumerate}[font=\normalfont]

\item \label{lem:structure_of_polytope_group:injection}
The homomorphism
\[
\xi \colon \calp_{\IZ}(H) \to\prod_{\phi \in \hom_{\IZ}(H,\IZ)} \calp_{\IZ}(\IZ), \quad [P] - [Q] 
\mapsto \bigl(\calp_{\IZ}(\phi)([P]- [Q])\bigr)_{\phi}
\]
is injective;

\item \label{lem:structure_of_polytope_group:splitting}
The canonical short sequence of abelian groups 
\[0\, \to \,H \,\xrightarrow{i} \,\calp_{\IZ}(H) \,\xrightarrow{\pr} \, \calp_{\IZ}^{\Wh}(H)\, \to\, 0\] is split exact;

\item \label{lem:structure_of_polytope_group:divisibility}

The abelian groups $\calp_{\IZ}(H)$ and $\calp_{\IZ}^{\Wh}(H)$ are free.

\end{enumerate}
\end{lemma}
\begin{proof}~\eqref{lem:structure_of_polytope_group:injection} Consider $[P]- [Q]$ in
  $\calp_{\IZ}(H)$. Suppose that $[P] - [Q]$ is not the trivial element. Then the polytopes
  $P$ and $Q$ are different. Hence we can assume without loss of generality that there
  exists $q \in Q$ with $q \notin P$ (if not,  consider $[Q] - [P]$). By the Separating
  Hyperplane Theorem, see~\cite[Theorem~V.4 on page~130]{Reed-Simon(1980)}, there exists
  an $\IR$-linear map $\psi \colon \IR \otimes_{\IZ} H \to \IR$ and $r \in \IR$ such that
  $\psi(x) < r$ holds for all $x \in P $ and $\psi(q) > r$.  By continuity we can find a
  $\IZ$-linear map $\phi' \colon H \to \IQ$ such that the same holds for the $\IR$-linear
  map $\IR \otimes H \to \IR$ induced by $\phi$. Choose a $\IZ$-map $\phi \colon H \to  \IZ$ 
  such that for some natural number $n$ we have $n \cdot \phi' = \phi$. Then we have
  $\phi(P) < n \cdot r$ and $\phi(q) > n \cdot r$. This implies $\phi(P) \not=
  \phi(Q)$. Hence $\calp_{\IZ}(\phi)([P]- [Q]) \not=0$ and therefore $\xi([P] - [Q]) \not=   0$.  
  \\[1mm]~\eqref{lem:structure_of_polytope_group:splitting}
   We pick an identification of  $H$ with $\IZ^n$. We endow $\IZ^n$ with the lexicographical order.
  It is straightforward to verify that there exists a unique homomorphism
  \[ \calp_{\IZ}(H)=\calp_{\IZ}(\IZ^n)\to H=\IZ^n\]
  with the property that a polytope gets sent to the extreme point of $P$ of lowest order. 
  This is clearly a splitting of the map $\IZ^n=H\,\to\,  \calp_{\IZ}(H)=\calp_{\IZ}(\IZ^n)$.
  \\[1mm]~\eqref{lem:structure_of_polytope_group:divisibility} 
It follows from Example~\ref{ex:polytopes-in-r} and assertion~\eqref{lem:structure_of_polytope_group:injection} 
that $\calp_{\IZ}(H)$ embeds into a countable free abelian group, hence it is free abelian by \cite{Specker(1950)}. 
It follows from assertion~\eqref{lem:structure_of_polytope_group:splitting} that  
$\calp_{\IZ}^{\Wh}(H)$ is also free-abelian. Explicit bases of the free abelian groups are given by
 Funke~\cite{Funke(2016)}. 
\end{proof}


\subsection{The polytope homomorphism}
\label{subsec:The_polytope_homomorphism}
In the following we will always mostly work with groups that satisfy the Atiyah Conjecture. For the reader's convenience we recall the statement.

\begin{definition}[Atiyah Conjecture]
\label{def:Atiyah_Conjecture}
We say that a torsion-free group $G$ satisfies the \emph{Atiyah Conjecture} if for any
matrix $A \in M_{m,n}(\IQ G)$ the von Neumann dimension $\dim_{\caln(G)}(\ker(r_A))$
of the kernel of the $\caln(G)$-homomorphism
$r_A\colon \caln(G)^m \to \caln(G)^n$ given by right multiplication with $A$ is an integer.
\end{definition}

The precise statement of the Atiyah Conjecture is not relevant to us, what is important is that we have the following proposition which is  \cite[Lemma~10.39]{Lueck(2002)}.

\begin{proposition}\label{prop:group-ring-atiyah-conjecture}
Let $G$ be a group that is torsion-free and that  satisfies the Atiyah Conjecture. Then the rational closure $\calr(G)$ agrees with the division closure $\cald(G)$ of $\IZ G
\subseteq \calu(G)$ and $\cald(G)$ is a skew-field. 
\end{proposition}

From now on we suppose that $G$ is torsion-free, satisfies the Atiyah Conjecture and $H_1(G)_f$ is finitely generated.
In Friedl-L\"uck~\cite[Section~6.2] {Friedl-Lueck(2016l2+poly)} the main ingredients of the so called \emph{polytope homomorphism}
\begin{eqnarray}
& \IP \colon K_1^w(\IZ G) \to \calp_{\IZ}(H_1(G)_f)&
\label{polytope_homomorphism}
\end{eqnarray}
have been established. We briefly recall its definition for the reader's convenience. 

There is a homomorphism, well-defined by Lemma~\ref{lem:L2-acyclic_and_ZG_subseteq_calr(G)-contractible},
\begin{eqnarray}
\theta \colon K_1^w(\IZ G) \to K_1(\calr(G)), \quad [f] \mapsto [\id_{\calr(G)}  \otimes_{\IZ G} f].
\label{theta_colon_K_1w(ZG)_to_cald(G)}
\end{eqnarray}

Then by Proposition~\ref{prop:group-ring-atiyah-conjecture} the rational closure $\calr(G)$ agrees with the division closure $\cald(G)$ of $\IZ G
\subseteq \calu(G)$ and $\cald(G)$ is a skew-field. There is a Dieudonn\'e determinant for
invertible matrices over a skew field which takes values in the abelianization of the
group of units, see~\cite[Corollary~4.3 in page~133]{Silvester(1981)}.  Hence we obtain an
isomorphism
\begin{eqnarray}
{\det}_D \colon  K_1(\cald(G))
& \xrightarrow{\cong}  &
\cald(G)^{\times}_{\abel} := \cald(G)^{\times}/[\cald(G)^{\times},\cald(G)^{\times}]. 
\label{K_1(cald(G))_Dieudonne}
\end{eqnarray}
The inverse 
\begin{eqnarray}
& J_{\cald(G)} \colon \cald(G)^{\times}_{\abel}   \xrightarrow{\cong}  K_1(\cald(G)) &
\label{K_1(K)_Dieudonne_inverse}
\end{eqnarray}
sends the class of a unit in $D$ to the class of the corresponding $(1,1)$-matrix.

Next we want to define a homomorphism
\begin{eqnarray}
\IP' \colon \cald(G)^{\times}_{\abel} \to \calp_{\IZ}(H_1(G)_f).
\label{P_prime_colon_cald_to_p_Z}
\end{eqnarray}
We denote by $K$ the kernel of $\pr\colon G\to H_1(G)_f$.  Choose a map of sets 
$s \colon H_1(G)_f \to G$ with $\pr \circ s = \id_{H_1(G)_f}$.  Then
there is an isomorphism
\[
\widehat{j_s} \colon T^{-1} (\cald(K) \ast_s H_1(G)_f) \xrightarrow{\cong} \cald(G),
\]
where $T^{-1} (\cald(K) \ast_s H_1(G)_f)$ is the Ore localization of the integral domain
given by the crossed product $\cald(K) \ast_s H_1(G)_f$ with respect to the
multiplicative set $T$ of non-trivial elements.  Consider an element 
$u \in \cald(K) \ast_s H_1(G)_f$ with $u \not= 0$.  We can write 
$u =\sum_{h \in H} u_h \cdot h \in \cald(K) \ast_s H_1(G)_f$ for appropriate elements $u_h \in \cald(K)$.  
Define an integral polytope $P(u) \subseteq \IR \otimes_{\IZ} H_1(G)_f$ by the convex hull of the 
subset $\{h \in H_1(G)_f \mid u_h \not = 0\}$ of $H_1(G)_f$. 
 (This may be viewed as a non-commutative version of the Newton polytope
of a polynomial in several variables.)
This construction is compatible with
the multiplication in $\cald(K) \ast_s H_1(G)_f$ and the Minkowski sum, namely, for 
$u,v \in \cald(K) \ast_s H_1(G)_f$ with $u,v \not= 0$ we get $P(uv) = P(u) + P(v)$. Thus we
obtain a homomorphism of abelian groups
\[
\left(T^{-1} (\cald(K) \ast_s H_1(G)_f)\right)^{\times}  \to \calp_{\IZ}(H_1(G)_f), 
\quad uv^{-1} \mapsto [P(u)] - [P(v)].
\]
If we compose it with the isomorphism coming from $\widehat{j_s}$ and take into account that
$\calp_{\IZ}(H_1(G)_f)$ is abelian, we get  the desired well-defined homomorphism $\IP'$ announced
in~\eqref{P_prime_colon_cald_to_p_Z}. 

The polytope homomorphism~\eqref{polytope_homomorphism} is defined to be the composite
\[
\IP \colon K_1^w(\IZ G) \xrightarrow{\theta} K_1(\cald(G)) 
\xrightarrow{{\det}_D} \cald(G)^{\times}_{\abel} 
\xrightarrow{\IP'} \calp_{\IZ}(H_1(G)_f)
\]
of the homomorphisms defined in~\eqref{theta_colon_K_1w(ZG)_to_cald(G)},~%
\eqref{K_1(cald(G))_Dieudonne}, and~\eqref{P_prime_colon_cald_to_p_Z}.

One easily checks that the polytope homomorphism~\eqref{polytope_homomorphism}
induces homomorphisms denoted by the same symbol $\IP$
\begin{eqnarray}
\IP \colon \widetilde{K}_1^w(\IZ G) & \to & \calp_{\IZ}(H_1(G)_f);
\label{polytope_homomorphism_reduced}
\\
\IP \colon \Wh^w(G) & \to & \calp_{\IZ}^{\Wh}(H_1(G)_f).
\label{polytope_homomorphism_Wh}
\end{eqnarray}

Sending the class of an integral polytope $P$ to the class of the integral polytope 
$-P = \{-p \mid p \in P\}$, yields involutions
\begin{eqnarray*}
 \ast \colon \calp_{\IZ}(H_1(G)_f) & \xrightarrow{\cong} &\calp_{\IZ}(H_1(G)_f),
\label{involution_on_calp_Z(H_1(G)_f)}
\\
\ast \colon \calp_{\IZ}^{\Wh}(H_1(G)_f) & \xrightarrow{\cong} &\calp_{\IZ}^{\Wh}(H_1(G)_f).
\label{involution_on_calp_Z,Wh(H_1(G)_f)}
\end{eqnarray*}
Consider any group homomorphism $w \colon G \to \{ \pm 1 \}$.  Equip $\IZ G$ with the
involution of rings sending $\sum_{g \in G} r_g \cdot g$ to 
$\sum_{g \in G} r_g \cdot w(g) \cdot g^{-1}$ and $K_1^w(\IZ G)$, $\widetilde{K}_1^w(\IZ G)$, and 
$\Wh^w(G)$ with the induced involutions. One easily checks

\begin{lemma} \label{lem:polytope_homomorphisms_and_involutions}
The  polytope homomorphisms introduced
in~\eqref{polytope_homomorphism}, \eqref{polytope_homomorphism_reduced},
and~\eqref{involution_on_calp_Z,Wh(H_1(G)_f)} are compatible with the involutions defined
above.
\end{lemma}

If $f \colon G \to K$ is an injective group homomorphism, then the following diagram
\begin{eqnarray}
&
\xymatrix@!C=12em@R=0.6cm {\Wh^w(G) \ar[r]^-{\Wh^w(f)} \ar[d]_{\IP^G}
& \Wh^w(K) \ar[d]^{\IP^K}
\\
\calp_{\IZ}^{\Wh}(H_1(G)_f) \ar[r]_-{\calp_{\IZ}^{\Wh}(H_1(f)_f)} 
& \calp_{\IZ}^{\Wh}(H_1(K)_f)
}
&
\label{naturality_of_the_polytope-homomorphism}
\end{eqnarray}
commutes and  is compatible with the involutions.

\begin{remark}[Computational complexity] \label{rem:computational_complexity} The main
  problem when one wants to compute the image of an element in $\Wh^w(G)$ under the polytope
  homomorphism $\bfP^G  \colon \Wh^w(G) \to \calp_{\IZ}^{\Wh}(H_1(G)_f) $
  is that for an $(n,n)$-matrix over $\IZ G$ it is usually very hard to figure out the
  corresponding unit in $\cald(G)^{\times}$ since the Dieudonn\'e determinant is not at all
  easy to compute.  See also~\cite[Remark~6.24]{Lueck(2015twisting)}. The situation is
  easy if $n = 1$, as exploited in
  Subsection~\ref{subsec:Torsion-free_one-relator_groups_with_two_generators}.
\end{remark}


\subsection{The $L^2$-torsion polytope}
\label{subsec:The_L2-torsion_polytope}

\begin{definition}[The $L^2$-torsion polytope]
  \label{def:The_L2-torsion_polytope}
  Let $G$ be a torsion-free group satisfying the Atiyah Conjecture. Suppose that $H_1(G)_f$ is finitely generated.
  Let $X$ be a free finite $G$-$CW$-complex which is $L^2$-acyclic. Then we define its
  \emph{$L^2$-torsion polytope}
  \[
   P(X;G) \in \calp_{\IZ}^{\Wh}(H_1(G)_f)
   \]
    to be the image of the \emph{negative} \label{sign-convention} of  the universal $L^2$-torsion $\rho^{(2)}_u(X;\caln(G)) \in \Wh^w(G)$
    introduced in Definition~\ref{def:universal_L2-torsion_for_G-CW-complexes} under the
    polytope homomorphism introduced in~\eqref{polytope_homomorphism_Wh}, i.e,
    \[
    P(X;G) = \IP^G\bigl(-\rho^{(2)}_u(X;\caln(G))\bigr).
    \]
    We take a minus sign in the definition above in order to get nicer formulas when relating $P(X;G)$ to the 
    Thurston norm and the dual Thurston polytope, see 
    Theorem~\ref{the:The_L2-torsion_for_universal_coverings_and_the_Thurston_seminorm} 
    and Theorem~\ref{the:dual_Thurston_polytope_and_the_L2-torsion_polytope_new}.

    If $X$ is an $L^2$-acyclic connected finite $CW$-complex, we abbreviate
    \[
    P(\widetilde{X}) := P(\widetilde{X};\caln(\pi_1(X)) \in \calp_{\IZ}^{\Wh}(H_1(X)_f).
    \]
  \end{definition}
  
  Now there are obvious analogues of the
  Theorems~\ref{the:Main_properties_of_the_universal_L2-torsion}
  and~\ref{the:Main_properties_of_the_universal_L2-torsion_for_universal_coverings} which
  are sometimes simpler to state.  As an illustration we go through a few examples for
  $P(\widetilde{X})$ for  $L^2$-acyclic connected finite $CW$-complexes $X$ and $Y$.
  Recall that we assume that the fundamental group is torsion-free and satisfies
  the Atiyah Conjecture.

  \begin{enumerate}

  \item (Homotopy invariance)
    If $X$ and $Y$ are simple homotopy equivalent or if $X$ and $Y$ are homotopy
    equivalent and $\pi_1(X)$ satisfies the $K$-theoretic Farrell-Jones Conjecture, then the image of $P(\widetilde{X})$
   under the isomorphism $\IR \otimes_{\IZ} H_1(X)_f \xrightarrow{\cong} \IR \otimes_{\IZ} H_1(Y)_f$  induced by $f$
   is $P(\widetilde{Y})$;

   \item ($S^1$-actions)
   Let $X$ be a connected finite $S^1$-$CW$-complex.
Suppose that for one and hence all $x \in X$ the map $\pi_1(S^1,1) \xrightarrow{\pi_1(\ev_x,1)} \pi_1(X,x)$
is injective, where $\ev_x \colon S^1 \to X$ sends $z$ to $z \cdot x$. 
Define the $S^1$-orbifold Euler characteristic of $X$ by
\[
\chi^{S^1}_{\orb}(X) \,\,=\,\, \smsum{n \ge 0}{} (-1)^n \cdot \smsum{e \in I_n}{} \smfrac{1}{|S^1_e|},
\]
where $I_n$  is the set of open $n$-dimensional $S^1$-cells of $X$ and for such an $S^1$-cell
$e \in I_n $ we denote by $S^1_e$ the isotropy group of any point in $e$.

Then $\widetilde{X}$ is $L^2$-acyclic and we get
\[
P(\widetilde{X}) =  
\chi^{S^1}_{\orb}(X) \cdot \calp_{\IZ}^{\Wh}(\ev_x )([J]),
\]
where the integral polytope $J \subseteq \IR \otimes_{\IZ} \IZ$ is the polytope given by 
$[0,1] \subseteq \IR = \IR \otimes_{\IZ} \IZ$.

\item (Seifert and graph manifolds)
  An analogous formula holds for Seifert manifolds with infinite fundamental group, 
  compare Remark~\ref{rem:graph_manifolds}.

 \end{enumerate}


\subsection{The $L^2$-torsion and the Thurston seminorm}
\label{subsec:The_L2-torsion_and_the_Turston_norm}
Let $H$ be a finitely generated torsion-free abelian group. Let $P \subseteq \IR \otimes_{\IZ} H$
be a polytope. It defines a seminorm on $\hom_{\IZ}(H,\IR) = \hom_{\IR}(\IR \otimes_{\IZ} H,\IR)$ by
\begin{eqnarray}
 \|\phi \|_P 
&:= & \tmfrac{1}{2}
\sup\{\phi(p_0) - \phi(p_1) \mid p_0, p_1 \in P\}.
\label{seminorm_of_a_polytope}
\end{eqnarray}
It is compatible with the Minkowski sum, namely, for two integral polytopes $P,Q \subseteq \IR \otimes_{\IZ} H$
we have
\begin{eqnarray}
 \|\phi \|_{P+ Q} 
&= & 
 \|\phi \|_{P}  +  \|\phi \|_{Q}.
\label{seminorm_of_a_polytope_is_additive_in_polytope}
\end{eqnarray}
Put
\begin{multline}
\calsn(H) 
 := 
\{f \colon \hom_{\IZ}(H;\IR) \to \IR \mid \text{there exist integral polytopes}
\\
\; P \; \text{and} \; Q\; \text{in}\; \IR \otimes_{\IZ} H \;\text{with}\; f =  \|\; \|_P -  \|\; \|_Q\}.
\label{calsn(H)}
\end{multline}
This becomes an abelian group by $(f-g)(\phi) = f(\phi) - g(\phi)$ because 
of~\eqref{seminorm_of_a_polytope_is_additive_in_polytope}. Again because 
of~\eqref{seminorm_of_a_polytope_is_additive_in_polytope}
we obtain an epimorphism  of abelian groups
\begin{equation}
\sn \colon \calp_{\IZ}^{\Wh}(H) \to \calsn(H) 
\label{norm_homomorphism}
\end{equation}
by sending $[P] - [Q]$ for
two polytopes $P,Q \subseteq \IR \otimes_{\IZ} H$ to the function
\[
\hom_{\IZ}(H,\IR) \to \IR, \quad \phi \mapsto  \|\phi \|_P -  \|\phi \|_Q.
\]

\begin{theorem}[The $L^2$-torsion and the Thurston seminorm]
\label{the:The_L2-torsion_and_the_Thurston_seminorm}
Let $M $ be an admissible  $3$-manifold which is not homeomorphic to $S^1 \times D^2$
and is not a closed graph manifold. We write $\pi=\pi_1(M)$. 
Then there is a  virtually finitely generated free abelian group $\Gamma$,
and a factorization $\pr_M \colon \pi \xrightarrow{\alpha} \Gamma \xrightarrow{\beta} H_1(M)_f$
of the canonical projection $\pr_M \colon \pi \to H_1(\pi)_f$ into epimorphisms,  such that the following holds:

Consider a torsion-free  group $G$ which satisfies the Atiyah Conjecture and for which $H_1(G)_f$ is finitely generated, and
any factorization of $\alpha \colon \pi \to \Gamma$ into group homomorphisms
$\pi \xrightarrow{\mu} G \xrightarrow{\nu} \Gamma$. Let $\overline{M} \to M$ be the $G$-covering associated to $\mu$. 
 Let $\phi \colon H_1(G)_f \to \IZ$ be any group homomorphism.
Then $\overline{M}$ is $L^2$-acyclic
and its universal $L^2$-torsion $\rho^{(2)}_u(\overline{M};\caln(G))$ is sent under the composite
\[
\IP\IN \colon \Wh^w(G) \xrightarrow{\IP} \calp_{\IZ}^{\Wh}(H_1(G)_f) \xrightarrow{\calp_{\IZ}^{\Wh}(H_1(\beta \circ \nu)_f)}
\calp_{\IZ}^{\Wh}(H_1(M)_f)
\xrightarrow{\sn} \calsn(H_1(M)_f)
\]
of the homomorphisms $\IP$, $\calp_{\IZ}^{\Wh}(H_1(\beta \circ \nu)_f)$, and
$\sn$ defined in~\eqref{polytope_homomorphism_Wh}, \eqref{calp_Z,Wh(f)}, and~\eqref{norm_homomorphism}
to the element given by the negative of half the Thurston seminorm
\[
\hom_{\IZ}(H_1(M)_f,\IR) \to \IR, \quad \phi \mapsto - \tmfrac{1}{2} x_M(\phi).
\]
\end{theorem}

\begin{proof}
Before we start with the proof we need to introduce some notation. We write $\pi=\pi_1(M)$. Given a homomorphism $\mu\colon \pi\to G$ and a homomorphism $\psi\colon G\to \IZ$ we denote by $\overline{M}$ the cover of $M$ corresponding to $\mu$, we write $C_*=C_*(\overline{M})$ and we define
\[ 
\chi^{(2)}(M;\mu,\psi)
\,\, = \,\, \chi^{(2)}\bigl((\caln(K) \otimes_{\IZ K} i^*C_*;\caln(K)\bigr),
\]
where  $K$ is the kernel of $\psi$ and $i\colon K\to G$ is the inclusion.

Now can go ahead with the proof.   By~\cite[Theorem~0.4]{Friedl-Lueck(2016l2+poly)} there exists a virtually finitely
  generated free abelian group $\Gamma$, and a factorization 
  $\pr_M \colon \pi_1(M)   \xrightarrow{\alpha} \Gamma \xrightarrow{\beta} H_1(M)_f$ of the canonical 
  projection  $\pr_M \colon \pi \to H_1(\pi)_f$ into epimorphisms, such that the following holds for
  any $\phi\in H^1(M;\IZ)=\hom(H_1(M)_f;\IZ)$ and any torsion-free group $G$ satisfying
  the Atiyah Conjecture and any factorization of $\alpha \colon \pi \to \Gamma$ into group
  homomorphisms $\pi \xrightarrow{\mu} G \xrightarrow{\nu} \Gamma$:
\begin{eqnarray}
\label{equ:chi-two-equals-thurston-norm}
\chi^{(2)}(M;\mu,\phi \circ \beta \circ \nu)
& = & 
- x_M(\phi),
\end{eqnarray}

Now let $\mu\colon \pi\to G$ be as above. 
Let $K$ be the kernel of $\phi \circ \beta \circ \nu$. We  use the notation introduced in the beginning of the proof. We start out with proving the claim
that
\begin{eqnarray}
\IP\IN\big(\rho^{(2)}_u(C_*)\big)(\phi)
& = & \tmfrac{1}{2}
\chi^{(2)}\bigl((\caln(K) \otimes_{\IZ K} i^*C_*;\caln(K)\bigr). 
\label{sn_circ_P(rho_u)_and_chi}
\end{eqnarray}
provided that $\phi \circ \beta \circ \nu$ is surjective.

In order to prove the claim we first consider an  $(n,n)$-matrix $A$ over $\IZ G$ which becomes invertible over $\cald(G)$.
It defines a class $[A] \in \Wh^w(G)$ by Lemma~\ref{lem:L2-acyclic_and_ZG_subseteq_calr(G)-contractible}
since our hypothesis that $G$ satisfies the Atiyah Conjecture implies by 
Proposition~\ref{prop:group-ring-atiyah-conjecture} that  $\cald(G) = \calr(G)$.
 We conclude from~\cite[Lemma~6.12 and Lemma~6.16]{Friedl-Lueck(2016l2+poly)} 
that we get in the notation of~\cite{Friedl-Lueck(2016l2+poly)}
\begin{eqnarray*}
\IP\IN([A])(\phi)
& = & -\tmfrac{1}{2}
\dim_{\cald(K)}\bigl(\coker\bigl(r_A \colon \cald(K)_{t}[u^{\pm 1}]^n \to \cald(K)_{t}[u^{\pm 1}]^n\bigr)\bigr).
\end{eqnarray*}
If $\el(r_A)$ denotes the $\IZ G$-chain complex concentrated in
dimension $0$ and $1$ with first differential $r_A \colon \IZ G^n \to
\IZ G^n$ and $i \colon K \to G$ is the inclusion, then we conclude
from~\cite[Theorem~3.6~(4)]{Friedl-Lueck(2016l2+poly)}
\begin{multline*}
\chi^{(2)}\bigl(\caln(K) \otimes_{\IZ K} i^*\el(r_A);\caln(K)\bigr)
\\
=\dim_{\cald(K)}\bigl(\coker\bigl(r_A \colon \cald(K)_{t}[u^{\pm 1}]^n \to \cald(K_{t}[u^{\pm 1}]^n\bigr)\bigr).
\end{multline*}
This implies
\begin{equation*}
\IP\IN([A])(\phi) 
\,\,=\,\, 
(\sn \circ \IP)\bigl(\rho^{(2)}_u(\el(r_A)\bigr)(\phi)
 \,\,=\,\,\tmfrac{1}{2} 
\chi^{(2)}\bigl(\caln(K) \otimes_{\IZ K} i^*\el(r_A);\caln(K)\bigr).
\end{equation*}
We conclude from Remark~\ref{rem:universal_property_of_the_universal_L2-torsion}  
that for any $L^2$-acyclic  finite based free $\IZ G$-chain complex $C_*$ we get
\begin{eqnarray*}
\IP\IN(\rho^{(2)}_u(C_*))(\phi)
& = & \tmfrac{1}{2}
\chi^{(2)}\bigl((\caln(K) \otimes_{\IZ K} i^*C_*;\caln(K)\bigr),
\end{eqnarray*}
since $C_* \mapsto \chi^{(2)}\bigl((\caln(K) \otimes_{\IZ K} i^*C_*;\caln(K)\bigr)$ defines 
an additive $L^2$-torsion invariant with values in $\IR$.
This concludes the proof of~\eqref{sn_circ_P(rho_u)_and_chi}.

We now turn to the actual proof of the theorem.  If we
apply the above claim to $C_* = C_*(\overline{M})$ and if we combine
the resulting equality with (\ref{equ:chi-two-equals-thurston-norm})
we see that
\begin{eqnarray*}
  \IP\IN\bigl(\rho^{(2)}_u(\overline{M};\caln(G))\bigr)(\phi)
  & = & 
  - \tmfrac{1}{2} x_M(\phi).
\end{eqnarray*}
provided that $\phi$ is surjective, since the surjectivity of $\phi$ implies the surjectivity of
$\phi \circ \beta \circ \nu$.

Both maps
\[
\IP\IN\bigl(\rho^{(2)}_u(\overline{M};\caln(G))\bigr)\mbox{ and } x_M \colon \hom_{\IZ}(H_1(M)_f,\IR)\to \IR
\]
are continuous since seminorms are continuous maps,  and satisfy
\begin{eqnarray*}
\IP\IN\bigl(\rho^{(2)}_u(\overline{M};\caln(G))\bigr)(r \cdot \phi) 
& = & 
|r| \cdot \IP\IN\bigl(\rho^{(2)}_u(\overline{M};\caln(G))\bigr)(\phi);
\\
- x_M(r \cdot \phi)
& = &
|r| \cdot (- x_M(\phi)),
\end{eqnarray*}
for $r \in \IR$ and $\phi \in \hom_{\IZ}(H_1(M)_f,\IR)$.
Hence we get for every  $\phi \in  \hom_{\IZ}(H_1(M)_f,\IR)$
\begin{eqnarray*}
\IP\IN\bigl(\rho^{(2)}_u(\overline{M};\caln(G))\bigr)(\phi)
& = & 
-\tmfrac{1}{2} x_M(\phi),
\end{eqnarray*}
since $\hom_{\IZ}(H_1(M)_f,\IQ)$ is dense in $\hom_{\IZ}(H_1(M)_f,\IR)$ and
for any non-trivial $\psi \in \hom_{\IZ}(H_1(M)_f,\IQ)$ there exists an epimorphism 
$\phi \colon H_1(M)_f \to \IZ$ and a rational number $r$ with $r \cdot \phi = \psi$.
 This finishes the proof of
Theorem~\ref{the:The_L2-torsion_and_the_Thurston_seminorm}.
\end{proof}

\begin{theorem}[The $L^2$-torsion for universal coverings and the
  Thurston seminorm]
  \label{the:The_L2-torsion_for_universal_coverings_and_the_Thurston_seminorm}
  Let $M$ be an admissible $3$-manifold which is not homeomorphic to  $S^1\times D^2$.  
  Suppose  $\pi_1(M)$ satisfies the Atiyah
  Conjecture.
  
Then the $L^2$-torsion polytope $P(\widetilde{M})$ of
Definition~\ref{def:The_L2-torsion_polytope} is sent under the homomorphism $\sn \colon
\calp_{\IZ}^{\Wh}(H_1(M)_f) \to \calsn(H_1(M)_f)$ to half of the Thurston seminorm $x_M$.
\end{theorem}

We point out that for ``almost all'' 3-manifolds $M$ the fundamental group $\pi_1(M)$ satisfies the Atiyah Conjecture. More precisely, $\pi_1(M)$ satisfies the Atiyah Conjecture   if $M$ is not a closed graph
   manifold or if $M$ is a closed graph manifold which admits
 a  Riemannian metric of non-positive sectional curvature.$)$
We refer to~\cite[Theorem~3.2]{Friedl-Lueck(2016l2+poly)} for details.

\begin{proof}
  If $M$ is not a graph manifold, the claim follows from
  Theorem~\ref{the:The_L2-torsion_and_the_Thurston_seminorm}. 
  (Note that here we use the sign convention in the definition of $P(\widetilde{M})$ that 
we introduced in the beginning of Section~\ref{subsec:The_L2-torsion_polytope}.) The case of a graph
  manifold is handled analogously using~\cite[Theorem~2.14]{Friedl-Lueck(2016l2+poly)}.
\end{proof}

\begin{remark} \label{rem:pairing_and_introduction}
The pairing~\eqref{pairing_wh_otimes_H1_to_Z} is given by the homomorphism
\[
\Wh^w(G) \xrightarrow{\IP} \calp_{\IZ}^{\Wh}(H_1(G)_f) \xrightarrow{\sn} \calsn(H_1(G)_f)
\]
namely, an element $x \otimes \phi \in \Wh^w(G) \otimes H^1(G)$ is sent to the evaluation
at $\phi \colon G \to \IZ$ of the element in $\calsn(H_1(G)_f)$ given by the image of
$\rho^{(2)}_u(C_*)$ under this homomorphism. The same argument as appearing in the proof
of~\eqref{sn_circ_P(rho_u)_and_chi} together with the formula 
$\chi^{(2)}(C_*;\caln(G),k\cdot \phi) = k \cdot \chi^{(2)}(C_*;\caln(G),\phi)$ for $k \in \IZ$ show the claim in
Subsection~\ref{subsec:twisted_L2-Euler_characteristic} that for an $L^2$-acyclic finite
free $\IZ G$-chain complex $C_*$ and an element $\phi \in H^1(G) $ the image of
$\rho^{(2)}_u(C_*;\caln(G)) \otimes \phi$ under the pairing above is
$\chi^{(2)}(C_*;\caln(G),\phi)$.
\end{remark}

In Theorem~\ref{the:The_L2-torsion_for_universal_coverings_and_the_Thurston_seminorm} we
just saw that the polytope $P(\widetilde{M})$ determines the Thurston norm of $M$. In the
coming sections we will prove a more precise statement, namely that the polytope
$P(\widetilde{M})$ agrees, up to translation, with the dual of the Thurston norm ball.


\subsection{Seminorms and compact  convex subsets}
\label{subsec:Seminorms_and_compact_convex_subsets}

Let $V$ be a finite-dimensional real vector space. We write $V^*=\hom(V,\IR)$. 
Given any subset $X$ of $V$, we define its dual $X^*$ to be
\begin{eqnarray}
X^* & = & \{ \phi\in V^*\,|\, \phi(v)\leq 1\; \mbox{for all}\; v\in X\}.
\label{dual_of_X}
\end{eqnarray}
Given a compact convex subset $X \subseteq V$ we use the definition of~\eqref{seminorm_of_a_polytope} to define a seminorm
\begin{equation*}
 \|\; \|_X \colon V^* \to [0,\infty), \quad \phi \mapsto  \|\phi \|_X := \tmfrac{1}{2}\sup\{\phi(x_0) - \phi(x_1) \mid x_0,x_1 \in X\}.
\end{equation*}
We define $(-X)$ to be the compact convex subset $\{-x \mid x \in X\}$.
The  Minkowski sum $X + (-X)$ is again a compact convex subset and we get
$ \|\; \|_X =  \|\;  \|_{-X}$ and $2 \cdot  \|\; \|_X  =  \|\; \|_{X + (-X)}$.

Given a seminorm $s$ on $V$, we assign to it its unit ball
\[
B_s := \{v \in V \mid s(v) \le 1\}
\]
and denote by $B_s^*$ the associated dual. A straightforward argument shows that we have the equality
\begin{equation}
B_s^* =   \{\phi \in V^* \mid \phi(v) \le s(v) \; \text{for all} \; v \in V\}.
\label{Q_s}
\end{equation}

In the sequel we will identify $V = V^{**}$ by the canonical isomorphism.

\begin{lemma} \label{lem:self_dual}
If $X$ is a closed convex subset containing $0$, then under the identification $V = V^{**}$ we have $X = X^{**}$. 
\end{lemma}

\begin{proof} It is straightforward to see that $X\subset X^{**}$. We now show the reverse inclusion. 
Consider $y \in V$ with $y \notin X$. By the Separating Hyperplane Theorem,
  see~\cite[Theorem~V.4 on page~130]{Reed-Simon(1980)}, we can find $\psi\in V^*$ and 
  $r  \in\IR$ such that $\psi(x) < r$ holds for all $x \in X$ and $\psi(y) > r$.  Since $0$ is
  contained in $X$ and since $\psi(0) = 0$ we deduce that $r > 0$. Define $\phi := r^{-1} \cdot
  \psi \in V^*$.  Then $\phi(x) \le 1 $ for $x \in X$ and $\phi(y) > 1$. This implies $y \notin X^{**}$.
\end{proof}

\begin{lemma} \label{lem:semi_norms_and_compact_subsetes}
 Let $s$ be a seminorm on $V$ and $X \subseteq V$ be a compact convex subset.
 Then 

\begin{enumerate}[font=\normalfont]

\item \label{lem:semi_norms_and_compact_subsetes:B_s_compact}
The convex set $B_s$ is compact if and only if $s$ is a norm;

\item \label{lem:semi_norms_and_compact_subsetes:B_sast_compact}
$B_s^*$ is convex and compact;

\item \label{lem:semi_norms_and_compact_subsetes:s_and_sup}
For any $v\in V$ we have
\[
s(v)\, = \,\frac{1}{\sup\{r \in [0,\infty) \mid rv \in B_s\}};
\]

\item \label{lem:semi_norms_and_compact_subsetes:B_s_dual_of_B_sast}
We have $B_s = (B_s^*)^*$;

\item \label{lem:semi_norms_and_compact_subsetes:comparing_norms}
We have $\|\;\|_{B_s^*} = s$;

\item \label{lem:semi_norms_and_compact_subsetes:Minkowski_and_norm_of_dual_ball}
We have $X + (-X)  = (B_{ \|\; \|_X})^*$.

\end{enumerate}
\end{lemma}
\begin{proof}~\eqref{lem:semi_norms_and_compact_subsetes:B_s_compact}
This is obvious.
\\[1mm]~\eqref{lem:semi_norms_and_compact_subsetes:B_sast_compact}
It is straightforward to see that $B_s^*$ is convex and it follows from the description \eqref{Q_s} that $B_s^*$ is compact.
\\[1mm]~\eqref{lem:semi_norms_and_compact_subsetes:s_and_sup}
Consider $r \in (0,\infty)$ and $v \in V$ with $s(v) \not= 0$. 
Then we get $rv \in B_s \Longleftrightarrow r \le s(v)^{-1}$ and hence
$\sup\{r \in [0,\infty) \mid rv \in B_s\} \le s(v)^{-1}$. This implies
$s(v) \le  \frac{1}{\sup\{r \in [0,\infty) \mid rv \in B_s\}}$. Since for $v \in V$ with $s(v) \not= 0$ we have
$s(s(v)^{-1} \cdot v) \in B_s$, the claim follows.
\\[1mm]~\eqref{lem:semi_norms_and_compact_subsetes:B_s_dual_of_B_sast}
This follows from Lemma~\ref{lem:self_dual}.
\\[1mm]~\eqref{lem:semi_norms_and_compact_subsetes:comparing_norms}
We compute for $v \in V$ 
\begin{eqnarray*}
 \|v \|_{B_s^*} 
& = & 
\tmfrac{1}{2}\sup\{\phi_0(v)- \phi_1(v) \mid \phi_0, \phi_1 \in B_s^*\}
\\
& = & 
\tmfrac{1}{2}\sup\{\phi_0(v) - \phi_1(v) \mid \phi_i \in V^*, \phi_i(w) \le s(w) \;\text{for all} \; w \in V\; \text{and}\; i =0,1\}
\\
& = & 
\tmfrac{1}{2}\sup\{\phi(v) \mid \phi \in V^*, \phi(w) \le s(w) \;\text{for all} \; w \in V\}
\\
& & \hspace{20mm} +
\tmfrac{1}{2}\sup\{\phi(-v) \mid \phi \in V^*, \phi(w) \le s(w) \;\text{for all} \; w \in V)\}.
\end{eqnarray*}
The Hahn-Banach  Theorem, see~\cite[Theorem~III.5 on page~75]{Reed-Simon(1980)},  implies for all $v \in V$
\[
\sup\{\phi(v) \mid \phi \in V^*, \phi(w) \le s(w) \;\text{for all} \; w \in V\} = s(v).
\]
Since $s(v) = s(-v)$, assertion~\eqref{lem:semi_norms_and_compact_subsetes:comparing_norms} follows.
\\[1mm]~\eqref{lem:semi_norms_and_compact_subsetes:Minkowski_and_norm_of_dual_ball}
Since $0$ is contained in $X + (-X)$,  Lemma~\ref{lem:self_dual} implies
\[
X + (-X) = (X + (-X))^{**}.
\]
We get directly from the definitions 
\begin{eqnarray*}
(X + (-X))^{**} 
& = & 
\bigl\{v \in V \mid \phi(v) \le 1 \; \text{for all}\; \phi \in X + (-X)^* \}
\\ 
& = & 
\bigl\{v \in V \mid \phi(v) \le 1 \; \text{for all}\; \phi \in V^*
\\
& & \hspace{3mm} 
\;\text{satisfying}\; \phi(w) \le 1 \; \text{for all}\; w \in X + (-X)\}
\\ 
& = & 
\bigl\{v \in V \mid \phi(v) \le 1 \; \text{for all}\; \phi \in V^*
\\
& & \hspace{3mm} 
\;\text{satisfying}\; \supp\{\phi(w_0) - \phi(w_1)  \mid w_0,w_1 \in X + (-X)\} \le 2\bigr\}
\\
& = & 
\bigl\{v \in V \mid \phi(v) \le 1 \; \text{for all}\; \phi \in V^* \;\text{satisfying}\;  \|\phi\|_{X + (-X)} \le 2\bigr\}
\\
& = & 
\bigl\{v \in V \mid \phi(v) \le 1 \; \text{for all}\; \phi \in V^* \;\text{satisfying}\;  \|\phi\|_X \le 1\bigr\}
\\
& = & 
\bigl\{v \in V \mid \phi(v) \le 1 \; \text{for all}\; \phi \in B_{ \|\; \|_X}\}
\\ 
& = & 
(B_{ \|\; \|_X})^*.
\end{eqnarray*}
This finishes the proof  of Lemma~\ref{lem:semi_norms_and_compact_subsetes}.
\end{proof}


\subsection{The dual Thurston polytope}
\label{subsec:The_dual_Thurston_polytope}

Now let $M$ be a compact oriented  $3$-manifold. 
In the sequel we will identify $\IR\otimes_{\IZ} H_1(M)_f = H_1(M;\IR)$ and $H^1(M;\IR) = H_1(M;\IR)^*$ and
$V = V^{**}$ by the obvious isomorphisms. We refer to 
\begin{equation}
B_{x_M}\;  := \; \{\phi\in H^1(M;\IR)\,|\, x_M(\phi)\leq 1\}
\label{Thurston_ball}
\end{equation}
as the \emph{Thurston norm ball} and we refer to 
\begin{equation}
T(M)^*:=B_{x_M}^*\subset (H^1(M;\IR))^*= H_1(M;\IR)
\label{dual_Thurston_polytope} 
\end{equation}
as  the \emph{dual Thurston polytope}.  Explicitly by \eqref{Q_s} we have
\[
T(M)^* = \{v \in H_1(M;\IR) \mid \phi(v) \le x_M(\phi) \; 
\text{for all} \; \phi \in H^1(M;\IR)\}.
\]
Thurston~\cite[Theorem~2 on page~106 and first paragraph on
page~107]{Thurston(1986norm)} has shown that $T(M)^*$ is an integral polytope.

\begin{theorem}[The dual Thurston polytope and the $L^2$-torsion polytope]
\label{the:dual_Thurston_polytope_and_the_L2-torsion_polytope_new}
Let $M$ be an admissible  $3$-manifold which is not homeomorphic to $S^1 \times D^2$.
 Suppose that $\pi_1(M)$ satisfies the Atiyah Conjecture.  Then
\begin{eqnarray*}
[T(M)^*] & = &  2 \cdot P(\widetilde{M})\,\, \in \,\, \calp_{\IZ}^{\Wh}(H_1(M)_f).
\end{eqnarray*}
\end{theorem}

Here we recall that we had pointed out after the statement of
Theorem~\ref{the:The_L2-torsion_for_universal_coverings_and_the_Thurston_seminorm} that the fundamental group $\pi_1(M)$ of ``almost all'' 3-manifolds satisfies the Atiyah Conjecture.

The proof of  Theorem~\ref{the:dual_Thurston_polytope_and_the_L2-torsion_polytope_new} will require the remainder of this section.


\subsection{The dual Thurston polytope and the $L^2$-torsion polytope}
\label{subsec:The_dual_Turston_polytope_and_the_L2-torsion_polytope}

Recall that we have defined an involution $\ast \colon \calp_{\IZ}(H) \to \calp_{\IZ}(H)$
by sending $[P] - [Q]$ to $[-P] - [-Q]$. It induces an involution
$\ast \colon \calp_{\IZ}^{\Wh}(H) \to \calp_{\IZ}^{\Wh}(H)$.

Next we define a homomorphism of abelian groups
\begin{equation}
\poly \colon \calsn(H) \to \im\bigl( \id + \ast \colon \calp_{\IZ}(H) \to \calp_{\IZ}(H)\bigr)
\label{poly}
\end{equation}
by sending an element in $\calsn(H)$ represented by $ \|\; \|_{P} -  \|\;  \|_Q$ for integral
polytopes $P$ and $Q$ in $\IR \otimes_{\IZ} H$ to the element $[P] + [-P] - [Q] -
[-Q]$. This is well-defined because
of~\eqref{seminorm_of_a_polytope_is_additive_in_polytope} and 
Lemma~\ref{lem:semi_norms_and_compact_subsetes}~%
\eqref{lem:semi_norms_and_compact_subsetes:Minkowski_and_norm_of_dual_ball}
which implies $ \|\;  \|_P =  \|\; \|_Q
\Longleftrightarrow P + (-P) = Q + (-Q)$ for integral polytopes $P$ and $Q$ in 
$\IR \otimes_{\IZ} H$.

\begin{lemma}\label{lem:poly_circ_sn}\
\begin{enumerate}[font=\normalfont]
\item \label{lem:poly_circ_sn:sn_circ_ast_is_sn}
The map $\sn \colon \calp_{\IZ}^{\Wh}(H) \to \calsn(H)$  defined in~\eqref{norm_homomorphism}  satisfies
$\sn \circ \ast = \sn$;
\item \label{lem:poly_circ_sn:composite}
We have the following commutative diagram
\[
\xymatrix{
\calp_{\IZ}(H) \ar[d]^{\pr} \ar[rrd]^{\id + \ast}
& & 
\\
\calp_{\IZ}^{\Wh}(H)  \ar[r]^-{\sn}
&
\calsn(H) \ar[r]^-{\poly}
&
\im\bigl( \id + \ast \colon \calp_{\IZ}(H) \to \calp_{\IZ}(H)\bigr)
}
\]
where $\pr$ is the canonical projection;

\item \label{lem:poly_circ_sn:sn_on_fixed_point_set}
The map $\sn$ induces an injective map
\[
\sn \colon \calp_{\IZ}^{\Wh}(H)^{\IZ/2} \to \calsn(H)
\]
whose cokernel is annihilated by multiplication with $2$,
and a surjective map
\[
\overline{\sn} \colon \IZ \otimes_{\IZ[\IZ/2]} \calp_{\IZ}^{\Wh}(H) \to \calsn(H)
\]
whose kernel is annihilated by multiplication with $2$.
\end{enumerate}
\end{lemma}
\begin{proof}~\eqref{lem:poly_circ_sn:sn_circ_ast_is_sn}
This follows from $ \|\; \|_P =  \|\; \|_{-P}$.
\\[1mm]~\eqref{lem:poly_circ_sn:composite}
This follows from the definitions.
\\[1mm]~\eqref{lem:poly_circ_sn:sn_on_fixed_point_set}
By definition of $\calsn(H)$ the map $\sn \colon \calp_{\IZ}^{\Wh}(H)\to \calsn(H)$ is surjective.
Hence also the map $\overline{\sn}$ is surjective. Consider $x \in \calp_{\IZ}^{\Wh}(H)^{\IZ/2}$
with $\sn(x) = 0$. Choose $y \in \calp_{\IZ}(H)$ with $ \pr(y) = x$. We get 
\[
y + \ast (y) = \poly \circ \sn \circ \pr(y) = \poly \circ \sn(x) = 0.
\]
We get in $\calp_{\IZ}^{\Wh}(H)$
\[
0 = \pr(y + \ast (y)) = x + \ast(x) = 2 \cdot x.
\]
This implies that the kernel of $\sn \colon \calp_{\IZ}^{\Wh}(H)^{\IZ/2} \to \calsn(H)$ is annihilated by multiplication with $2$.
Since $\calp_{\IZ}^{\Wh}(H)$ is torsionfree, 
see Lemma~\ref{lem:structure_of_polytope_group}~\eqref{lem:structure_of_polytope_group:divisibility},
the map $\sn \colon \calp_{\IZ}^{\Wh}(H)^{\IZ/2} \to \calsn(H)$  is injective.
Since the kernel and the cokernel of the map
\[j \colon  \calp_{\IZ}^{\Wh}(H)^{\IZ/2}  \to \IZ \otimes_{\IZ[\IZ/2]} \calp_{\IZ}^{\Wh}(H), \quad z \mapsto  1 \otimes z
\]
are annihilated by multiplication with $2$ and $\overline{\sn} \circ j = \sn|_{\calp_{\IZ}^{\Wh}(H)^{\IZ/2}}$,
Lemma~\ref{lem:poly_circ_sn} follows.
\end{proof}

Now we can give the proof of Theorem~\ref{the:dual_Thurston_polytope_and_the_L2-torsion_polytope_new}.

\begin{proof}[Proof of Theorem~\ref{the:dual_Thurston_polytope_and_the_L2-torsion_polytope_new}]
We conclude from Theorem~\ref{the:The_L2-torsion_for_universal_coverings_and_the_Thurston_seminorm}
 that the $L^2$-polytope $P(\widetilde{M}) \in \calp_{\IZ}^{\Wh}(H_1(M)_f)$ is sent under
$\sn \colon \calp_{\IZ}^{\Wh}(H_1(M)_f) \to \calsn(H_1(M)_f)$ to $-\tmfrac{1}{2}x_M$.
By 
Lemma~\ref{lem:semi_norms_and_compact_subsetes}~\eqref{lem:semi_norms_and_compact_subsetes:comparing_norms} we have
\[ \sn([T(M)^*])\,\,=\,\,\|-\|_{{T(M)^*}}\,\,=\,\,\|-\|_{B^*_{x_M}}\,\,=\,\,x_M.\]
This shows  that $(-2) \cdot P(\widetilde{M})$ and 
$[T(M)^*]$ are sent by  the homomorphism $\sn \colon \calp_{\IZ}^{\Wh}(H_1(M)_f) \to \calsn(H_1(M)_f)$ to $x_M$.

By construction $T(M)^* = -T(M)^*$. Hence $[T(M)^*]$ lies in $\calp_{\IZ}^{\Wh}(H_1(M)_f)^{\IZ/2}$.

We conclude
\[
\ast\bigl(\rho^{(2)}_u(\widetilde{M})\bigr) = \rho^{(2)}_u(\widetilde{M},\widetilde{M}|_{\partial M})
\]
from the version of Poincar\'e duality for compact manifolds with boundary, 
see Theorem~\ref{the:Main_properties_of_the_universal_L2-torsion}~(8).
From the additivity of the universal $L^2$-torsion, see 
Lemma~\ref{lem:Additivity_of_universal_L2-torsion},
 we get
\[ 
\rho^{(2)}_u(\widetilde{M},\widetilde{M}|_{\partial M})=
\rho^{(2)}_u(\widetilde{M}) -  \rho^{(2)}_u(\widetilde{M}|_{\partial M}).
\]
Since $\partial M$ is a union of incompressible tori, we conclude 
$\rho^{(2)}_u(\widetilde{M}|_{\partial M}) = 0$ from
Example~\ref{exa:torus}. This implies
\[ 
\ast\bigl(\rho^{(2)}_u(\widetilde{M},\widetilde{M}|_{\partial M})\bigr) =
\rho^{(2)}_u(\widetilde{M}).
\]
Recall that the $L^2$-polytope $P(\widetilde{M})$ is defined 
be the image of $- \rho^{(2)}_u(\widetilde{M})$ under the polytope homomorphism
$\IP \colon \Wh^w(\pi_1(M)) \to \calp_{\IZ}^{\Wh}(H_1(M)_f)$. 
We conclude from Lemma~\ref{lem:polytope_homomorphisms_and_involutions} that 
$P(\widetilde{M}) \in \calp_{\IZ}^{\Wh}(H_1(M)_f)^{\IZ/2}$.

Since both $[T(M)^*]$ and $(-2) \cdot P(\widetilde{M})$ lie in
$\calp_{\IZ}^{\Wh}(H_1(M)_f)^{\IZ/2}$ and have the same image under $\sn \colon
\calp_{\IZ}^{\Wh}(H_1(M)_f) \to \calsn(H_1(M)_f)$, namely $-x_M$, we conclude from
Lemma~\ref{lem:poly_circ_sn}~\eqref{lem:poly_circ_sn:sn_on_fixed_point_set}
that $[T(M)^*] = (-2) \cdot P(\widetilde{M}) $ holds in $\calp_{\IZ}^{\Wh}(H_1(M)_f)$. 
This finishes the proof of Theorem~\ref{the:dual_Thurston_polytope_and_the_L2-torsion_polytope_new}.
\end{proof}


\typeout{------------------------------   Section 4: Examples -----------------------------------}

\section{Examples} 
\label{sub:examples}


\subsection{$L^2$-acyclic groups }
\label{subsec:L2-acyclic_groups}

Let $G$ be a group with a finite model for $BG$ which is $L^2$-acyclic, i.e., all its
Betti numbers $b_n^{(2)}(G;\caln(G))$ vanish. Furthermore suppose that its Whitehead group
$\Wh(G)$ is trivial. (The Farrell-Jones
  Conjecture, which is known for a large class of groups containing for instance
  hyperbolic groups, CAT(0)-groups, lattices in almost connected Lie groups and solvable
  groups, implies the vanishing of $\Wh(G)$ for torsionfree $G$.)  Then we get by
Theorem~\ref{the:Main_properties_of_the_universal_L2-torsion}~%
\eqref{the:Main_properties_of_the_universal_L2-torsion:G-homotopy_invariance} a
well-defined invariant
\begin{eqnarray}
\rho^{(2)}_u(G) 
&:= & 
\rho_u^{(2)}(EG;\caln(G)) \in \Wh^w(G).
\label{rho(2)_u(G)}
\end{eqnarray}
If $G$ satisfies the Atiyah Conjecture, we can apply
Definition~\ref{def:The_L2-torsion_polytope} to $EG$ and obtain an element
\begin{eqnarray}
P(G) 
& := & 
P(EG;G) \in \calp_{\IZ}^{\Wh}(H_1(G)_f)
 \label{L2-polytope_of_a_group}
\end{eqnarray}
which is the image of $-\rho^{(2)}_u(G)$ under the polytope homomorphism in~\eqref{polytope_homomorphism_Wh}.
There is an obvious sum formula for amalgamated products coming from 
Theorem~\ref{the:Main_properties_of_the_universal_L2-torsion_for_universal_coverings}~(2).
Namely, if we have injective group homomorphisms $G_0 \to G_i$ for $i =1,2$ and
$G_i$ has a finite model for $BG_i$ and is $L^2$-acyclic for $i =0,1,2$, then $G = G_1 \ast_{G_0} G_2$  has a finite model
for $BG$, is $L^2$-acyclic and we get in $\Wh^w(G)$
\begin{eqnarray}
\rho^{(2)}_u(G) & = & (j_1)_*\rho^{(2)}_u(G_1) +  (j_2)_*\rho^{(2)}_u(G_2) -  (j_0)_*\rho^{(2)}_u(G_0),
\label{sum_formula_for_rho(2)_u(G)}
\end{eqnarray} 
where $j_i \colon G_i \to G$ is the inclusion. There are also obvious analogues of the 
finite covering formula, the product formula and the statement about fibrations of 
Theorem~\ref{the:Main_properties_of_the_universal_L2-torsion}.


\subsection{Torsion-free one-relator groups with two generators}
\label{subsec:Torsion-free_one-relator_groups_with_two_generators}

Let $G$ be a torsion-free one-relator group with two generators which is not the free
group.  Choose any presentation $\langle x,y \mid R\rangle$ with two generators and one
relation $R$. Let $X$ be the associated presentation complex which has one zero-cell, two
$1$-cells, one for each generator $x$ and $y$, and one $2$-cell which is attached to the
$1$-skeleton which is the wedge of two copies of $S^1$ according to the word $R$. Then
$\pi_1(X)$ is isomorphic to $G$ and $X$ is a model for $BG$, 
see~\cite[Chapter~III~\S\S 9--11]{Lyndon-Schupp(1977)}. The finite based free $\IZ
G$-chain complex of the universal covering $\widetilde{X}$ is given in terms of Fox
derivatives by
\[
\IZ G \xrightarrow{\bigl(\tmfrac{\partial R}{\partial x} \,  \tmfrac{\partial R}{\partial y} \bigr)}
\IZ G^2 \xrightarrow{\footnotesize{\begin{pmatrix} x-1\\ y-1 \end{pmatrix}}} \IZ G
\]
It is known that $b_n^{(2)}(EG;\caln(G)) = 0$ holds for all $n \ge 0$, 
see~\cite[Theorem~4.2]{Dicks-Linnell}. There is an obvious short based exact 
sequence of finite based free $\IZ G$-chain complexes
\[0 \to \Sigma \el(r_{\frac{\partial R}{\partial x}} \colon \IZ G \to \IZ G) \to C_*(\widetilde{X}) \to 
\el(r_{y-1} \colon \IZ G \to \IZ G) \to 0
\]
Since $C_*(\widetilde{X})$ and $\el(r_{y-1} \colon \IZ G \to \IZ G)$ are $L^2$-acyclic, the complex
  $\Sigma \el(r_{\frac{\partial R}{\partial x}})$ is also $L^2$-acyclic. Hence we get in $\Wh^w(G)$
\[
\rho^{(2)}_u(C_*(\widetilde{X})) = - [r_{\frac{\partial R}{\partial x}} \colon \IZ G \to \IZ G] +
[r_{y-1} \colon \IZ G \to \IZ G].
\]
Waldhausen~\cite[page~249-250]{Waldhausen(1978genfreeI+II)} has proved that $\Wh(G)$ vanishes.
Hence $\rho^{(2)}_u(C_*(\widetilde{X}))$ depends only on the
homotopy type of $X$ and hence is an invariant of $G$ as a group which  we denote by  $\rho^{(2)}_u(G)$.
Hence $\rho^{(2)}_u(G)$ is independent of the presentation and satisfies
\[
\rho^{(2)}_u(G) = - \left[r_{\frac{\partial R}{\partial x}} \colon \IZ G \to \IZ G\right] +
[r_{y-1} \colon \IZ G \to \IZ G].
\]
Now suppose that $G$ satisfies the Atiyah Conjecture. Then we have the
polytope homomorphism
\[
\IP \colon \Wh^w(G) \to \calp_{\IZ}^{\Wh}(H_1(G)_f)
\]
introduced in~\eqref{polytope_homomorphism}. For an element $u = \sum_{g \in G} r_g \cdot
g \in \IZ G$ define its support by
\[
\supp_G(u) = \{g \in G \mid r_g \not= 0\} \subseteq G.
\]
Let $P(u) \subseteq \IR \otimes_{\IZ} H_1(G)_f$ be the integral polytope which is the convex hull
of the subset $\pr(\supp_G(u)) \subseteq H_1(G)_f$. Then we conclude from the definitions
\begin{equation}
\label{equ:polytope-two-generator}
\IP\bigl(-\rho^{(2)}_u(G) \bigr)\,\, =\,\, \left[P\left(\smfrac{\partial R}{\partial x}\right)\right] - [P(y-1)].
\end{equation}
The polytope $P(y-1)$ is the convex hull of the two points $0$ and $\pr(y) \in H_1(G)$.
Consider the example $G = \IZ^2 = \langle x,y| xyx^{-1}y^{-1}\rangle$. We compute
\begin{eqnarray*}
\smfrac{\partial xyx^{-1}y^{-1}}{\partial x}  
& = & 1- y
\end{eqnarray*}
and hence we get in $\calp_{\IZ}^{\Wh}(\IZ^2)$
\[
\rho^{(2)}_u(\IZ^2) = - [P(1-y)] +
[P(y-1)] = 0.
\]
This is consistent with Example~\ref{exa:torus}, where we have shown
$\rho^{(2)}_u(\widetilde{T^2};\caln(\IZ^2)) = 0$.

\begin{remark}\label{rem:Friedl-Tillman}
  With the same notation as above, the first author and
  Tillmann~\cite{Friedl-Tillmann(2015)} assigned to such a presentation 
  $\pi=\langle x,y  \mid R\rangle$ in an elementary way a polytope $P(\pi)$ in $H_1(G;\IR)$.  If $G$
  satisfies the Atiyah Conjecture it follows from (\ref{equ:polytope-two-generator}) and
  \cite[Proposition~3.5]{Friedl-Tillmann(2015)} that $P(\pi)$ and $P(G)$ represent the
  same element of $\calp_{\IZ}^{\Wh}(H_1(G)_f)$. In particular this shows that $P(\pi)$ is
  an invariant of the underlying group $G$ and not just of the presentation. This
  proves~\cite[Conjecture~1.2]{Friedl-Tillmann(2015)}, provided that the Atiyah Conjecture
  holds for $G$.
  
  If $\pi$ is furthermore the fundamental group of a 3-manifold, then Theorem~\ref{the:dual_Thurston_polytope_and_the_L2-torsion_polytope_new} 
says in particular that $2\cdot P(\pi)=[T(M)^*]$, this recovers the main theorem of~\cite{Friedl-Schreve-Tillmann(2015)}.
\end{remark}


\subsection{Group endomorphisms}
\label{subsec:group_automorphisms}

Let $f \colon G \to G$ be a monomorphism of a group $G$. Suppose that
$G$ admits a finite model for $BG$.  Then we can consider the mapping
torus $T_{Bf}$ of the induced map $Bf \colon BG \to BG$ induced by
$f$. It is straightforward to see that it is a finite model for the classifying spaces 
of the HNN-extension $G \ast _f$ associated to $f$.
(Indeed, the only statement that needs verification is that the higher
homotopy groups of $T_{Bf}$ are zero. We denote by
$\widetilde{T_{Bf}}$ the obvious infinite cyclic cover. It suffices to
show that its higher homotopy groups are zero. But each map $S^k\to
\widetilde{T_{Bf}}$ lies in a compact subset of $\widetilde{T_{Bf}}$,
in particular it lies in a subspace given by finitely many mapping
tori of $Bf\colon BG\to BG$ glued together, but such subspaces are
homotopy equivalent to $BG$ and hence aspherical. See also~\cite[Section~2]{Lueck(1994b)}.)
By~\cite[Theorem~2.1]{Lueck(1994b)} the space $T_{Bf}$ is
$L^2$-acyclic.  Since the simple homotopy type of $T_{Bf}$ is
independent of the choice of $BG$ and $Bf$, see
\cite[(22.1)]{Cohen(1973)}, we get well-defined invariants.
\begin{eqnarray}
  \rho^{(2)}_u(f) & := & \rho^{(2)}_u(\widetilde{T_{Bf}}) \quad \in \Wh^w(G \ast _f);
  \label{rho(2)_u(f)_of_a_group_automorphism}
  \\
  P(f)&  :=  & P(\widetilde{T_{Bf}}) \quad \in \calp_{\IZ}^{\Wh}(H_1(G \ast _f)),
  \label{P(f)_of_a_group_automorphism}
\end{eqnarray}
where for~\eqref{P(f)_of_a_group_automorphism} we have to assume that $G \ast _f$
satisfies the
Atiyah Conjecture. (This is for example always satisfied if $G$ is a free group, see \cite[Section~2.9]{Funke-Kielak(2016)} for details.) 

We expect that $P(f)$ is a useful invariant, already for automorphisms of free groups 
$P(f)$ should contain some interesting information.


\typeout{-------------------------------------- References  ---------------------r------------------}



\end{document}